\documentclass{amsart}

\usepackage[T1]{fontenc}
\usepackage[utf8]{inputenc}
\usepackage{lmodern}

\usepackage{hyperref}
\usepackage{graphicx}
\usepackage[english,french]{babel}

\usepackage{amsmath, amsthm, amssymb, amscd}
\usepackage{amssymb,url,xspace}

\usepackage{enumitem}
\usepackage[all,cmtip]{xy}

\newtheorem{prop}{Proposition}[section]
\newtheorem{coro}[prop]{Corollaire}
\newtheorem{theo}[prop]{Théorème}
\newtheorem{lemm}[prop]{Lemme}

\newtheorem*{hyp}{Hypothèse $(\star)$}

\theoremstyle{definition}
\newtheorem{defi}[prop]{Définition}
\newtheorem*{nota}{Notation}

\theoremstyle{remark}
\newtheorem{rema}[prop]{Remarque}
\newtheorem{exem}[prop]{Exemple}
\newtheorem*{reme}{Remerciements}

\newcommand{\resp}{\emph{resp}.\xspace}

\newcommand{\cf}{\emph{cf.}\xspace}
\newcommand{\etc}{\emph{etc.}\xspace}

\newcommand{\Dec}{\operatorname{Dec}}

\newcommand{\pp}{\operatorname{p}}
\newcommand{\Mod}{\operatorname{Mod}}
\newcommand{\Top}{\operatorname{Top}}
\newcommand{\Hom}{\operatorname{Hom}}
\newcommand{\Parf}{\operatorname{Parf}}
\newcommand{\ner}{\operatorname{ner}}
\newcommand{\Ner}{\operatorname{Ner}}
\newcommand{\fpqc}{\operatorname{fpqc}}
\newcommand{\Zar}{\operatorname{Zar}}
\newcommand{\ob}{\operatorname{ob}}
\newcommand{\fl}{\operatorname{fl}}

\newcommand{\Sch}{\operatorname{Sch}}
\newcommand{\Spec}{\operatorname{Spec}}
\newcommand{\uSpec}{\operatorname{\underline{Spec}}}
\newcommand{\LL}{\operatorname{L}}
\newcommand{\BB}{\operatorname{B}}
\newcommand{\Db}{\operatorname{D^b}}
\newcommand{\DD}{\operatorname{D}}
\newcommand{\PP}{\operatorname{P}}

\newcommand{\Simp}{\operatorname{Simp}}
\newcommand{\dd}{\operatorname{d}}
\newcommand{\sss}{\operatorname{s}}
\newcommand{\NN}{\operatorname{N}}
\newcommand{\Cat}{\operatorname{Cat}}
\newcommand{\RR}{\operatorname{R}}
\newcommand{\st}{\operatorname{st}}

\newcommand{\Diag}{\operatorname{Diag}}
\newcommand{\pt}{\operatorname{pt}}

\newcommand{\CC}{\operatorname{C}}
\newcommand{\GSp}{\operatorname{GSp}}
\newcommand{\GL}{\operatorname{GL}}
\newcommand{\GG}{\operatorname{\mathbb{G}}}
\newcommand{\Res}{\operatorname{Res}}

\newcommand{\tr}{\operatorname{tr}}
\newcommand{\End}{\operatorname{End}}
\newcommand{\rg}{\operatorname{rg}}
\newcommand{\diag}{\operatorname{diag}}

\newcommand{\uSym}{\operatorname{\underline{Sym}}}
\newcommand{\Com}{\operatorname{Com}}
\newcommand{\mon}{\operatorname{mon}}

\newcommand{\SSSS}{\operatorname{S}}
\newcommand{\id}{\operatorname{id}}
\newcommand{\HH}{\operatorname{H}}
\newcommand{\lloc}{\operatorname{loc}}
\newcommand{\otL}{\operatorname{\mathop\otimes\limits^L}}

\newcommand{\Div}{\operatorname{Div}}
\newcommand{\Parfis}{\operatorname{Parf-is}}
\newcommand{\Pic}{\operatorname{Pic}}
\newcommand{\Picis}{\operatorname{Pic-is}}

\newcommand{\Fr}{\operatorname{Fr}}
\newcommand{\can}{\operatorname{can}}

\newcommand{\red}{\operatorname{red}}

\newcommand{\Gp}{\Gamma_0(p)}
\newcommand{\Gpp}{\Gamma_1(p)}
\newcommand{\QQ}{\mathbb{Q}}
\newcommand{\OF}{\mathcal{O}_F}
\newcommand{\Zp}{\mathbb{Z}_{(p)}}
\newcommand{\Qp}{\mathbb{Q}_p}
\newcommand{\Zps}{\mathbb{Z}_{(p)}^{\times}}
\newcommand{\Nn}{\mathbb{N}}
\newcommand{\ZZ}{\mathbb{Z}} 
\newcommand{\NNN}{\mathbb{N}}
\newcommand{\OS}{\mathcal{O}_S}
\newcommand{\OX}{\mathcal{O}_X}
\newcommand{\OY}{\mathcal{O}_Y}
\newcommand{\Afp}{\mathbb{A}_f^{p}}
\newcommand{\Fp}{\mathbb{F}_{p}}
\newcommand{\Fq}{\mathbb{F}_{q}}
\newcommand{\Zq}{\mathbb{Z}_{q}}
\newcommand{\Zpp}{\mathbb{Z}_{p}}

\newcommand{\LLL}{\mathcal{L}}
\newcommand{\Yq}{Y_{0,\Zq}}

\newcommand{\CCC}{\mathbb{C}}
\newcommand{\SSS}{\mathcal{S}}
\newcommand{\AAA}{\mathcal{A}}
\newcommand{\III}{\mathcal{I}}
\newcommand{\JJJ}{\mathcal{J}}
\newcommand{\NNNNN}{\mathcal{N}}

\newcommand{\Fqs}{\mathbb{F}_{q}^{\times}}
\newcommand{\Zqs}{\mathbb{Z}_{q}^{\times}}
\newcommand{\OO}{\mathcal{O}}

\newcommand{\AlGv}{_A\ell_G^{\vee}}

\newcommand{\AplGv}{_{A'}{\ell}_G^{\vee}}
\newcommand{\AllGv}{_A\underline{\ell}_G^{\vee}}

\newcommand{\lMKv}{\ell^{M,\vee}_K}

\newcommand{\llMK}{\underline{\ell}^{M}_K}
\newcommand{\uHom}{\underline{\Hom}}

\newcommand{\cpk}{\chi^{p^k}}
\newcommand{\Fqm}{\mathbb{F}_q^{\mon}}
\newcommand{\Am}{A^{\mon}}
\newcommand{\ALS}{A\otL\OS}

\newcommand{\ZM}{\ZZ[M]}

\newcommand{\CCCC}{\CC_{\centerdot}\text{--}(1)}
\newcommand{\CCd}{\CC_{\centerdot}}
\newcommand{\PPd}{\PP_{\centerdot}}
\newcommand{\Delred}{\Delta^2\text{--}^{\red}(-1)}

\newcommand{\ddd}{\Delta}
\newcommand{\ddo}{\Delta^{\circ}}
\newcommand{\ddp}{\overline{\Delta}'}
\newcommand{\ddb}{\overline{\Delta}}
\newcommand{\Topd}{\Top^{\circ}}
\newcommand{\nerd}{\ner^{\circ}}
\newcommand{\Xloc}{X_{\lloc}}
\newcommand{\XZar}{X_{\Zar}}

\newcommand{\Xloct}{\widetilde{X_{\lloc}}}
\newcommand{\XZart}{\widetilde{X_{\Zar}}}
\newcommand{\Yloc}{Y_{\lloc}}
\newcommand{\Yloct}{\widetilde{Y_{\lloc}}}

\newcommand{\Sloc}{S_{\lloc}}

\newcommand{\Sfpqc}{S_{\fpqc}}

\newcommand{\Nm}{\NNN\cup\{-1\}}

\newcommand{\Modz}{\Mod_{\ZZ}}

\newcommand{\testleft}{\leftarrow\!\shortmid}

\newcommand{\tim}[1]{\mathbin{\mathop{\times}\limits_{#1}}}

\begin{document}

\title{Modèle local des schémas de Hilbert-Siegel de niveau~$\Gpp$}

\author{Shinan Liu}
\address{Technishce Universität München, Zentrum Mathematik M11, Boltzmannstr. 3, 85748 Garching bei München, Deutschland/Germany}
\email{liushinan.thu@gmail.com}

\begin{abstract}
Nous construisons un modèle local pour les schémas de Hilbert-Siegel de niveau
$\Gpp$ sur $\Spec\Zq$
lorsque $p$ est non-ramifié dans le corps totalement réel,
où $q$ est le cardinal résiduel au-dessus de $p$.
Notre outil principal est une variante sur le petit site de Zariski
du complexe de Lie anneau-équivariant $\AllGv$ défini par Illusie
dans sa thèse,
où $A$ est un anneau commutatif
et $G$ est un schéma en $A$-modules.
Ce complexe nous permet de calculer le
complexe de Lie $\Fq$-équivariant d’un schéma en groupes de Raynaud, 
donc de relier
le modèle entier et le modèle local.\\

We construct a local model for 
Hilbert-Siegel moduli schemes with $\Gpp$-level bad reduction over $\Spec\Zq$,
where $p$ is a prime unramified in the totally real field
and $q$ is the residue cardinality over $p$.
Our main tool is a variant over the small Zariski site
of the ring-equivariant Lie complex $\AllGv$
defined by Illusie in his thesis,
where $A$ is a commutative ring 
and $G$ is a scheme of $A$-modules.
We use it to calculate the $\Fq$-equivariant Lie complex of a 
Raynaud group scheme,
then relate the integral model and the local model.
\end{abstract}

\maketitle

\vspace*{6pt}

\tableofcontents  

\section{Introduction}

Deux sujets sont abordés dans cet article.
Premièrement nous étudions la mauvaise réduction 
de certaines variétés de Shimura:
nous construisons un modèle entier et un modèle local
des variétés de Hilbert-Siegel de niveau $\Gpp$,
où $p$ est un nombre premier non-ramifié dans le corps totalement réel.
En même temps,
 nous définissons une variante du complexe de Lie
 anneau-équivariant d'un schéma en modules
construit par Illusie \cite{I}.
Cette redéfinition du complexe de Lie est notre outil principal 
pour la construction du modèle local,
mais nous pensons qu'elle est indépendamment intéressante,
et nous espérons qu'elle trouvera des applications
dans différents problèmes de déformation équivariante.

La mauvaise réduction des variétés de Shimura
a été étudiée par beaucoup d'auteurs.
Une attention toute particulière est portée 
aux niveaux parahoriques
parce qu'une large classe de modèles entiers 
est construite par Rapoport-Zink \cite{RZ}.
Un des problèmes centraux est alors le calcul
de la fonction zêta.
Le \emph{modèle local} est utile pour ce problème
car il possède le même faisceau de cycles proches---la
donnée locale clé de la fonction zêta.
La théorie des modèles locaux parahoriques
est établie par de Jong \cite{dJ}, Deligne-Pappas \cite{DP}, Rapoport-Zink \cite{RZ},
plus récemment par Pappas-Zhu \cite{PZ} en toute généralité 
via la théorie des groupes.

Au contraire, la théorie des modèles entiers 
et modèles locaux aux niveaux plus profonds
est moins développée à présent.
Dans le cas de niveau $\Gpp$
(c'est-à-dire que le sous-groupe de niveau est le radical pro-unipotent
d'un sous-groupe d'Iwahori),
des modèles entiers sont construits et étudiés 
par Pappas \cite{P1} pour les schémas de Hilbert-Blumenthal,
et par Haines-Rapoport \cite{HR}
pour les variétés de Shimura de type Harris-Taylor.
Plus récemment dans un travail de Haines et Stroh \cite{HS},
les auteurs ont étudié les variétés de Siegel de niveau $\Gpp$;
ils ont établi une théorie du modèle local pour ces variétés.
Les modèles locaux en niveau $\Gpp$ pour les variétés de Siegel
(et certaines variétés de Shimura unitaires)
sont aussi considérés par Shadrach \cite{Sh}.

Dans cet article, nous tentons également d'étudier des niveaux non-parahoriques.
Plus précisément nous généralisons le modèle local
de Haines-Stroh \cite{HS} aux variétés de Hilbert-Siegel.
Pour énoncer nos résultats, 
nous rappelons les modèles entiers définis dans \cite{HS}.
Soit $p$ un nombre premier, 
et $g$ un entier strictement positif.
Notons $Y_0/\Spec\Zpp$ le schéma de Siegel
de niveau $\Gp$.
Par définition $Y_0$ est le schéma de module des systèmes
$$
(A,H_1\subset...\subset H_g),
$$
où $A$ est un schéma abélien de genre $g$ (avec d'autres structures
qui garantissent la représentabilité),
et $H_i$ est un sous-groupe fini plat isotrope
de $A[p]$
tel que $\rg (H_i)=p^i$ pour tout $i$.
Dans \cite{HS}, les auteurs ont construit un modèle entier $Y_1/\Spec\Zpp$
de la variété de Siegel de niveau $\Gpp$.
Le schéma $Y_1$ est construit par la théorie d'Oort-Tate \cite{OT}:
en fait, 
les schémas en groupes $H_i/H_{i-1}$ sont de rang $p$
donc classifiés par \cite{OT}.
\footnote{La même idée est utilisée aussi dans \cite{HR} et \cite{Sh}
pour construire les modèles entiers.}

\begin{defi} \label{generateurOT}
Soit $S/\Spec\Zpp$ un schéma,
et $H/S$ un schéma en groupes d'Oort-Tate \cite{OT}.
D'après \cite{OT}, 
$H$ est isomorphe à
$$\uSpec_{\OS}(\uSym_{\OS}\LLL/(a-1)\LLL^{\otimes p})$$
pour un fibré en droites $\LLL$ sur $S$ et 
un élément $a$ de
$\Hom_{\OS}(\LLL^{\otimes p}, \LLL)$.
On dit que $(\LLL, a)$ est le \emph{paramètre d'Oort-Tate} de $H$.
On appelle un \emph{générateur d'Oort-Tate} de $H$
un élément~$z$ de $\Hom_{\OS}(\LLL,\OS)$
tel que $$z^{\otimes (p-1)}=a.$$
\end{defi}

\noindent
La définition de $Y_1$ par \cite{HS} est la suivante:
$Y_1$ 
est le schéma de module des systèmes 
$$(A,H_1\subset...\subset H_g,
z_1,...,z_g),$$
où $z_i$ est un générateur d'Oort-Tate de $H_i/H_{i-1}$.
\footnote{Il est également nécessaire de choisir des générateurs  
pour les duaux de Cartier $(H_i/H_{i-1})^{\vee}$, voir la définition~\ref{Yun}.}

Les définitions des schémas $Y_0$ et $Y_1$ se généralisent 
aux variétés de Hilbert-Siegel.
Dans ce cas,
les schémas en groupes $H_i/H_{i-1}$ sont classifiés par la théorie de Raynaud \cite{R},
et on a une définition de générateur 
similaire à la définition \ref{generateurOT}.
Pour être précis, on fixe quelques notations.
Soit $F/\QQ$ un corps totalement réel degré $n$, 
$\OF$ son anneau des entiers,
et $p$ un nombre premier \emph{non-ramifié} dans $F$.
Dans cette introduction et la plupart de cet article,
on suppose que $p$ est inerte
\footnote{Le cas non-ramifié général est traité dans \S 5.4.}
et on note $q=p^n$.
Le schéma de Hilbert-Siegel $Y_0/\Spec\Zpp$ de niveau $\Gp$ 
est défini exactement comme le précédent,
sauf qu'on suppose que $\rg H_i=q^i$
et on ajoute une  $\OF$-action sur
$A$ et $(H_i)$ avec la condition de Kottwitz.
Les schémas $H_i/H_{i-1}$ sont des schémas en $\Fq$-vectoriel de dimension $1$
et classifiés par \cite{R}.
\footnote{C'est vrai au moins sur la base $\Spec\Zq$.}
Dans la définition \ref{generateur} on va définir la notion 
de \emph{générateur de Raynaud}.
On obtient le modèle entier $Y_1/\Spec\Zq$ de niveau $\Gpp$
par le même problème de module que le paragraphe précédent,
sauf qu'ici $z_i$ est un générateur de Raynaud de $H_i/H_{i-1}$. 
Notre résultat principal sur les variétés de Shimura est le théorème suivant
(le théorème \ref{diagmodellocal}).

\begin{theo}\label{theoprin}
Avec les notations précédentes,
le schéma de Hilbert-Siegel $Y_1/\Spec\Zq$ 
admet un modèle local $M_1^+/\Spec\Zq$
défini explicitement par les définitions \ref{Mzeroplus}, \ref{Munplus}:
il existe un diagramme de modèle local
$$
Y_1 \xleftarrow{\pi'} Z_1^+ \xrightarrow{f'} M_1^+
$$
où les morphismes $\pi'$ et $f'$ sont lisses,
et $\pi'$ est surjectif.
\end{theo}

Dans la suite de cette introduction,
nous expliquons les idées pour construire  $M_1^+$ 
et démontrer le théorème \ref{theoprin}.
Comme on a remarqué 
le point clé est un complexe de Lie équivariant.
On explique d'abord comment le complexe de Lie intervient 
en considérant  le cas $F=\QQ$.
On note $M_0/\Spec\Zpp$ le modèle local de $Y_0$
défini par \cite{dJ} et \cite{RZ}.
Par la définition de $M_0$, la chaîne
$$
\omega_{A}\leftarrow \omega_{A/H_1}\leftarrow...\leftarrow\omega_{A/H_g}
$$
des algèbres de co-Lie de $A/H_i$ descend sur $M_0$.
On note $(\LLL_i, \alpha_i)$
le paramètre d'Oort-Tate de $H_i/H_{i-1}$,
 \cf la définition \ref{generateurOT}.
D'après la définition de $Y_1$,
on sait que pour construire le modèle local $M_1^+$ de $Y_1$ au-dessus de $M_0$,
il suffit de trouver le paramètre $\alpha_i$
dans le complexe $[\omega_{A/H_{i}}\to \omega_{A/H_{i-1}}]$.

L'idée de \cite{HS} pour répondre cette question
est inspirée par la théorie du degré des schémas en groupes finis plats 
de Fargues \cite{Far}.
On suppose que $S/\Spec\Zpp$ est un schéma 
et $(A,H_1\subset...\subset H_g)$ est un $S$-point de $Y_0$.
On note $H=H_i/H_{i-1}$, $B=A/H_{i-1}$ donc $B/H=A/H_i$,
et $(\LLL, a)$ le paramètre d'Oort-Tate de $H$.
On considère l'immersion fermée naturelle 
\begin{equation}\label{immersion}
i: H\hookrightarrow X=\uSpec_{\OS}(\uSym_{\OS}\LLL).
\end{equation}
On note $\III= (a-1)\LLL^{\otimes p}\subset \uSym_{\OS}\LLL$ l'idéal de $i$,
et $e$ la section neutre de $H$.
Par \cite{HS}, il existe un isomorphisme canonique de complexe
\begin{equation}\label{moddiff}
[\LLL^{\otimes p}\xrightarrow{a} \LLL]
\cong 
[e^*(\III/\III^2) \to e^*i^*\Omega_X].
\end{equation}
On note $\ell_H$ le complexe de co-Lie  de $H$ 
défini par \cite{I}.
Comme $i$ est une immersion régulière et $X$ est lisse,
par la théorie de complexe cotangent \cite{I} 
on obtient deux isomorphismes canoniques
\begin{equation}\label{lie}
[e^*(\III/\III^2) \to e^*i^*\Omega_X]
\cong\ell_H \cong
[\omega_{B/H}\to \omega_{B}]
\end{equation}
dans $\Db(\OS)$.
Le déterminant 
de complexe parfait (\cite{KM})
nous fournit,
d'après (\ref{moddiff}) et (\ref{lie}),
un isomorphisme de fibré en droites
\begin{equation}\label{det}
\det[\LLL^{\otimes p}\xrightarrow{a} \LLL]
\cong \det[\omega_{B/H}\xrightarrow{} \omega_{B}],
\end{equation}
qui commute avec les sections canoniques.
Le fibré à gauche est isomorphe à $\LLL^{\otimes (1-p)}$
avec la section canonique $a$,
donc on a trouvé le paramètre $a$ dans le complexe
$[\omega_{B/H}\to \omega_{B}]$.

Pour démontrer le théorème \ref{theoprin},
on va généraliser l'idée de \cite{HS} 
à un corps totalement réel $F$ quelconque.
On reformule la question de retrouver les paramètres de $H$
pour $F$ quelconque.
On note $n$ le degré $[F:\QQ]$.
Soit $S/\Spec\Zq$ un schéma,
et $H/S$ un schéma en groupes classifié par \cite{R},
\cf la définition \ref{groupeRaynaud}.
D'après \cite{R},
$H$ admet un paramètre 
$(\LLL_0,...,\LLL_{n-1}, a_0, ..., a_{n-1})$,
où les $\LLL_k$ sont des fibrés en droites sur $S$,
et $a_k$ appartient à $\Hom_{\OS}(\LLL_{k-1}^{\otimes p}, \LLL_k)$
pour tout $k$.
Soit 
$$
0\to H\to B \to B/H\to 0
$$
une résolution $\OF$-équivariante de $H$ par un schéma abélien 
$B/S$ muni d'une $\OF$-action.
La question est la suivante:
est-ce qu'on peut trouver les $a_k$ dans le complexe 
$[\omega_{B/H}\to \omega_{B}]$?

On essaie de généraliser l'idée de \cite{HS}.
Il se trouve que l'isomorphisme (\ref{moddiff})
se généralise bien à cette situation: 
il existe une immersion $\Fqs$-équivariante de $H$ qui généralise (\ref{immersion}),
et que les $a_k$ sont isomorphes aux \emph{composantes $\Fqs$-isotypiques}
de $[e^*(\III/\III^2) \to e^*i^*\Omega_X]$,
\cf le lemme \ref{lemmedepart}.
La difficulté est que le groupe $\Fqs$ n'agit pas naturellement 
sur $[\omega_{B/H}\to \omega_{B}]$,
donc a priori ce n'est pas évident comment trouver les $a_k$
via (\ref{lie}).

Nous proposons la stratégie suivante pour résoudre ce problème.
On note que dans (\ref{lie}) 
il existe une $\OF$-action naturelle sur $[\omega_{B/H}\to \omega_{B}]$,
et on peut considérer les composantes $\OF$-isotypiques
(car $S$ est défini sur $\Spec\Zq$).
On utilise systématiquement la théorie de complexe cotangent 
équivariant de Illusie \cite{I} VII
(on rappelle les définitions dans \S 4.2.1 avec plus de détails).
On note
$\Sfpqc$ le \emph{grand} site fpqc de $S$,
et $\otimes=\otimes_{\ZZ}$.
L'action du groupe $\Fqs$ sur $H$ définit
le complexe de Lie $\Fqs$-équivariant de $H$:
$$
\underline{\ell}_H^{\Fqs,\vee} \in 
\Db (\ZZ[\Fqs]\otimes \OO_{\Sfpqc}),
$$
où $\ZZ[\Fqs]$ est l'algèbre de groupe de $\Fqs$ sur $\ZZ$.
Il existe un isomorphisme 
$$
[e^*(\III/\III^2) \to e^*i^*\Omega_X]^{\vee}
\cong \underline{\ell}_H^{\Fqs,\vee}
$$
dans $\Db (\ZZ[\Fqs]\otimes \OO_{\Sfpqc})$,
qui induit des isomorphismes sur les composantes $\Fqs$-isotypiques.
D'autre part, la structure de $\OF$-module sur $H$
définit le complexe de Lie $\OF$-équivariant de~$H$:
$$
_{\OF}{\underline{\ell}}_H^{\vee}
\in \Db(\OF\otimes\OO_{\Sfpqc}).
$$
Il existe un autre isomorphisme
$$
_{\OF}{\underline{\ell}}_H^{\vee} \cong
[\omega_{B/H}\to \omega_{B}]^{\vee}
$$
dans $\Db(\OF\otimes\OO_{\Sfpqc})$,
qui induit des isomorphismes sur les composantes $\OF$-isotypiques.
La question de retrouver les paramètres $a_k$
dans $[\omega_{B/H}\to \omega_{B}]$
se réduit à trouver un isomorphisme entre les composantes isotypiques de 
$\underline{\ell}_H^{\Fqs, \vee}$ et 
$_{\OF}{\underline{\ell}}_H^{\vee}$.

Pour cela, on considère 
$$
_{\Fq}{\underline{\ell}}_H^{\vee}
\in \Db(\Fq\otL\OO_{\Sfpqc})
$$
le complexe de Lie $\Fq$-équivariant de $H$.
On veut définir les composantes isotypiques de $_{\Fq}{\underline{\ell}}_H^{\vee}$
et les relier aux composantes de 
$\underline{\ell}_H^{\Fqs, \vee}$ et 
$_{\OF}{\underline{\ell}}_H^{\vee}$
via les morphismes évidents
$$
\Fqs \hookrightarrow \Fq \twoheadleftarrow \OF.
$$
Or, comme $\Sfpqc$ est un grand site,
le faisceau $\OO_{\Sfpqc}$ n'est pas plat sur $\ZZ$,
donc $$\Fq\otL\OO_{\Sfpqc}$$
est un dg-anneau (\S 4.3), 
et ce n'est pas évident comment définir la notion de composante isotypique
dans $\Db(\Fq\otL\OO_{\Sfpqc})$.
\footnote{On peut bien sûr calculer $\Fq\otL\OO_{\Sfpqc}$
en utilisant la résolution  évidente
$0\to \Zq \xrightarrow{p}\Zq \to \Fq \to 0$.
Néanmoins, il ne nous semble pas facile d'en déduire
une notion de composante isotypique
qui soit compatible avec celle dans 
$\Db (\ZZ[\Fqs]\otimes \OO_{\Sfpqc})$.}
On a donc fait le choix de remplacer partout 
le grand site $\Sfpqc$ 
par  le  petit site de Zariski de $S$,
en vérifiant que les constructions de \cite{I} 
s'adaptent bien à ce nouveau cas.
\footnote{
Le choix du grand site dans \cite{I} n'est pas au hasard.
Le point est que le complexe $\AllGv$ sur le grand site
permet de décrire les déformations $A$-équivariantes de $G$,
mais avec le petit site ce n'est plus vrai.}
On obtient la proposition suivante (la proposition \ref{AlGv}),
qui est notre résultat technique clé.
\footnote{Nous croyons que cette redéfinition était déjà connue 
par Illusie.}
Le problème de définir les composantes isotypiques
de $_{\Fq}{\underline{\ell}}_H^{\vee}$ 
est alors résolu quand on prend $S=Y_0$.

\begin{prop}\label{algv}
Soit $A$ un anneau commutatif, 
$S$ un schéma \emph{plat}
\footnote{La démonstration de cette proposition sera valable
sans cette hypothèse de platitude. 
On l'ajoute ici simplement pour minimiser les discussions
sur les dg-anneaux.} 
sur $\Spec\ZZ$,
et $G$ un schéma en $A$-modules, 
plat et de présentation finie sur $S$.
On note ${\OO}_{\Sfpqc}$ (\resp $\OS$) 
l'anneau structurel du grand site fpqc
(\resp petit site de Zariski)
de $S$.
Alors la construction de \cite{I} sur le complexe 
$$\AllGv \in \Db(A\otL\OO_{\Sfpqc})$$
reste valable si l'on remplace $\OO_{\Sfpqc}$ par $\OS$.
On note
$$
\AlGv \in \Db(A\otimes\OS)
$$
l'objet obtenu par la reconstruction.
\end{prop}

On obtient aussi la proposition suivante 
(la proposition \ref{compatibilite})
sur la compatibilité des complexes de Lie équivariants.
Avec la proposition \ref{algv},
elle nous permet de relier les composantes de 
${\ell}_H^{\Fqs, \vee}$,
$_{\Fq}{{\ell}}_H^{\vee}$
et $_{\OF}{{\ell}}_H^{\vee}$,
donc répondre la question de trouver les paramètres de Raynaud,
\cf la proposition \ref{compatibiliteRaynaud},
le corollaire \ref{Raynaudabelien}.

\begin{prop}\label{compa}
Dans le contexte de la proposition \ref{algv},
on note $\Am$ le monoïde multiplicatif sous-jacent de $A$,
et ${\ell}_G^{\Am,\vee}$ le complexe de Lie $\Am$-équivariant de $G$.
On considère le foncteur
$$
\Com: \Db(A\otimes\OS)\to \Db(\ZZ[\Am]\otimes\OS)
$$
induit par le morphisme évident
$\ZZ[\Am]\to A$.
Alors il existe un isomorphisme canonique
$$
\Com(\AlGv)\cong {\ell}_G^{\Am,\vee}
$$
dans $\Db(\ZZ[\Am]\otimes\OS)$.
\end{prop}

Quelques remarques sur le complexe $\AllGv$ d'Illusie.
Une application principale de $\AllGv$ est de décrire 
les déformations d'un groupe de Barsotti-Tate,
\cf \cite{I2}.
La définition de $\AllGv$ dans \cite{I} est 
non-triviale et très compliquée:
elle dépend de façon essentielle
d'une idée topologique.
Cette idée est formulée dans \cite{I} VI
(voir aussi \S 4.2.3 de cet article)
et appelée la «stabilisation d'un foncteur non-additif»,
qui a été utilisée pour calculer
la homologie de Hochschild topologique,
\cf l'article de Breen \cite{B}.

Voici le plan de cet article.
Dans \S 2 nous rappelons les définitions des schémas $Y_0$ et $M_0$.
Dans \S 3 nous définissons le schéma $Y_1$ via la théorie de Raynaud.
Dans \S 4 nous démontrons les propositions \ref{algv} et \ref{compa}.
Nous rappelons dans \S 4.2 les définitions des complexes équivariants d'Illusie,
et nous expliquons pourquoi on peut les généraliser au site de Zariski.
Enfin, dans \S 5 nous rappelons le déterminant d'un complexe parfait
et démontrons le théorème \ref{theoprin}.

\begin{reme}
Je remercie mon directeur de thèse Benoît Stroh pour m'avoir proposé ce problème, 
et pour les discussions et son soutien depuis longtemps.
Je remercie sincèrement Monsieur Luc Illusie pour son attention à ce travail
et son encouragement.
Je remercie aussi Sophie Morel d'avoir corrigé une démonstration
dans une version préliminaire.
\end{reme}

\section{Schémas de Hilbert-Siegel de niveau $\Gp$}

\subsection{Le schéma $Y_0$}

Dans cette section, on définit le schéma de Hilbert-Siegel
de niveau $\Gp$.
Fixons d'abord les notations.
Soit $n$ un entier positif.
Soit $F$ un corps de nombres totalement réel de degré $n$
sur $\QQ$.
On note $\OF$ l'anneau des entiers de $F$.
Soit $p$ un nombre premier, non-ramifié dans $F$. 
On suppose d'abord que $p$ est \emph{inerte},
avec le cas général traité en \S 5.4.
Soit $g\geq 1$ un autre entier.

On note
$\langle \ ,\ \rangle: F^{2g} \times F^{2g} \to F$ 
la forme symplectique standard sur $F^{2g}$.
Elle est normalisée par 
$\langle e_i,e_{2g+1-j}\rangle=\delta_{i,j}$ 
pour tout $i,j\in\{1,...,2g\}$,
où $e_i$ est le $i$-ième élément dans la base canonique de $F^{2g}$.
Le groupe de similitudes symplectiques~~$\GSp_{2g, F}$ 
est un groupe algébrique sur $F$,
tel que pour toute $F$-algèbre $R$,
$$\GSp_{2g, F}(R)=
\{(h,c)\in \GL_{2g}(R)\times R^{\times}
|\langle hv ,hw \rangle =c \langle v,w \rangle,
\forall v, w \in R^{2g}\}.
$$
On définit le sous-groupe algébrique $G/\QQ$
de $\Res_{F/\QQ}\GSp_{2g,F}$ par
$$
G=\{(h,c)\in \Res_{F/\QQ}\GSp_{2g,F} |\ c \in \GG_{m,\QQ} \}.
$$

Soit $S$ un schéma, et $A,A'$ des schémas abéliens sur $S$. 
On appelle une quasi-isogénie $f: A\to A'$ au-dessus de $S$
une \emph{$\Zps$--isogénie}
si le degré de $f$ appartient à $\Zps$. 
Lorsque $A'=A^{\vee}$ est le dual de $A$,
on dit que $f$ est une \emph{$\Zps$--polarisation}
s'il existe de plus $N\in\Nn^*$ tel que $Nf$ est une polarisation.
On note que si $f$ est une $\Zps$-isogénie, 
alors elle induit un isomorphisme de schémas en groupes
$f: A[p]\to A'[p]$
sur les $p$-torsions.
On note $\End(A)$ l'anneau des endomorphismes de $A$.
On note $\omega_A=e^*\Omega_{A/S}$ l'algèbre de co-Lie de $A$,
où $e$ la section unité.

On note $(\text{Sch}/\Zp)$ (\resp $(\text{Ens})$) la catégorie des schémas sur $\Zp$ (\resp des ensembles).
On note $\mathbb{A}_f^{p}$ les adèles finies de $\QQ$ hors de $p$, 
et on fixe un sous-groupe ouvert compact $K^p$ de $G(\mathbb{A}_f^{p})$ suffisamment petit.
Dans cet article on utilise la notation $\otimes$ (\resp $\Hom$)
pour le bi-foncteur $\otimes_{\ZZ}$ (\resp $\Hom_{\ZZ}$),
sauf précisé autrement.

\begin{defi} \label{Yzero}

On définit le foncteur $Y_0:(\text{Sch}/\Zp)\to (\text{Ens})$, 
tel que pour tout schéma $S/\Zp$,
$Y_0(S)$ est égal à
l'ensemble des classes d'isomorphismes de 
$(A,\lambda,\iota,\bar{\eta},H_{\centerdot})$, où
\begin{enumerate}[label=--]

\item $A/S$ est un schéma abélien de genre $ng$;

\item $\lambda: A\to A^{\vee}$ est une $\Zps$--polarisation;

\item $\iota: \OF\otimes\Zp \to \End(A)\otimes \Zp$ 
est un morphisme d'anneau,
tel que pour tout~$a$ qui appartient à $\OF\otimes\Zp$, 
$\iota(a)$ commute avec $\lambda$. 
On suppose que localement sur~$S$, 
$\omega_A$ est un $\OF\otimes\OS$ -module libre
(avec l'action de $\OF$ induite par $\iota$).

\item $\bar{\eta}$ est une $K^p$-structure de niveau au sens de Kottwitz.
Plus précisément 
$\bar{\eta}$ est une $K^p$-orbite de similitudes symplectiques
$F$-linéaires $\eta: F^{2g}\otimes\Afp \to H_1^{et}(A_{\bar{s}}, \Afp)$ 
qui est aussi $\pi_1(S, \bar{s})$-invariante.
Ici $\bar{s}$ est un point géométrique de~~$S$,
et $H_1^{et}(A_{\bar{s}}, \Afp)$ est muni la forme symplectique
induite par $\lambda$.

\item $H_{\centerdot}=(H_1\subseteq...\subseteq H_g)$
est un drapeau de 
sous-groupes finis plats de $A[p]$ stables par l'action de $\OF$, 
tel que 
$\rg H_i=p^{ni}$ pour tout $i$,
et $H_g$ est $\lambda$-isotrope.
On suppose que localement sur $S$, 
$\omega_{A/H_i}$ est un $\OF\otimes\OS$ -module libre 
pour tout~~$i$.

\end{enumerate}
 
Un isomorphisme $f$ entre 
$(A,\lambda,\iota,\bar{\eta},H_{\centerdot})$ et 
$(A',\lambda',\iota',\bar{\eta'},H'_{\centerdot})$
est une $\Zps$--isogénie $f: A\to A'$, telle que

\begin{enumerate}[label=--]

\item sur chaque composante connexe de $S$,
on a $\lambda=r\cdot f^{\vee}\lambda' f$ pour une constante $r\in\Zps$;

\item pour tout $a\in\OF\otimes\Zp$,
$f\iota(a)=\iota'(a)f$;

\item $f_*\bar{\eta}=\bar{\eta'}$;

\item pour tout $i$,
$f: H_i \to H'_i$ 
est un isomorphisme de schémas en groupes.

\end{enumerate}

\end{defi}

Le foncteur $Y_0$ est défini dans \cite{RZ} definition 6.9.
Ici on utilise aussi le langage de \cite{L}.
On remarque que dans \cite{RZ},
la condition imposée sur $\omega_A$ (\resp $\omega_{A/H_i}$)
est la condition de déterminant de Kottwitz
(\cite{K} \S 5),
mais elle est équivalente à notre définition.
D'après \cite{RZ}, \cite{K},
le foncteur $Y_0$
est représentable par un schéma quasi-projectif sur $\Zp$.
On le note aussi par $Y_0$.
La fibre générique de $Y_0$ est la variété de Shimura 
pour le groupe $G$
de niveau $K_pK^p$
avec $K_p$ le sous-groupe d'Iwahori standard de 
$G(\QQ_p)\subset \GSp_{2g}(\QQ_q)$,
où $\QQ_q$ est l'extension non-ramifiée de $\QQ_p$
de degré $n$.
On appelle donc $Y_0$ 
le schéma de Hilbert-Siegel de niveau $\Gp$.

\subsection{Le modèle local $M_0$}
\subsubsection{}
Nous rappelons d'abord la notion de 
chaîne polarisée définie dans \cite{RZ} definition 3.14
pour les variétés de Hilbert-Siegel.
On donne une définition équivalente,
qui est similaire à la définition dans \cite{dJ}
pour les variétés de Siegel.
Soit $S$ un schéma sur $\Zp$.
Pour tout $\OS$-module $M$ localement libre de type fini,
on note $M^{\vee}$ le $\OS$-module $\underline{\Hom}_{\OS}(M, \OS)$.

\begin{defi} \label{chaine}
On appelle une 
chaîne de $\OF\otimes\OS$-modules polarisée,
ou simplement une \emph{chaîne polarisée} sur $S$,
la donnée
$$
M_{\centerdot}=(M_0 \leftarrow M_1 \leftarrow \ldots \leftarrow M_g, 
q_0, q_g),
$$
où 
\begin{enumerate}[label=--]
\item Pour $i=0,...,g$, 
$M_i$ est un faisceau de $\OF\otimes\OS$-modules
qui est localement sur $S$ libre de rang $2g$,
et le morphisme $M_{i-1} \leftarrow M_{i}$ est $\OF\otimes\OS$-linéaire.

\item Pour $j=0$ ou $g$, $q_j: M_j \times M_j \to \OS$ est une forme alternée parfaite, et $q_j(ax, y)=q_j(x, ay)$ pour tout $a\in \OF$ et $x,y\in M_j$.

\item Pour $j=0$ ou $g$, en identifiant $M_j$ avec $M_j^{\vee}$ via $q_j$, on obtient une chaîne $2g$-périodique 
$$
\ldots \leftarrow M_g \simeq M_g^{\vee} \leftarrow \ldots \leftarrow
M_0^{\vee} \simeq M_0 \leftarrow M_1 \leftarrow \ldots \leftarrow 
M_g \leftarrow \ldots
$$
On suppose que la composée des flèches dans chaque période est égale à 
la multiplication par $p$.

\item Pour $i=1,...,g$, localement sur $S$ le $\OF/p \otimes \OS$-module $M_{i-1}/M_{i}$ est libre de rang $1$.
\end{enumerate}
\end{defi}

\begin{exem} \label{chainestandard}
La chaîne polarisée standard sur $\Zp$ est la chaîne
$$
St_{\centerdot}=(St_0 \leftarrow St_1 \leftarrow \ldots \leftarrow St_g, q_0, q_g),
$$
où $St_i=\OF^{2g}\otimes\Zp$ pour tout $i$,
et pour $i=1,...,g$ 
la flèche $St_{i-1} \leftarrow St_{i}$ est définie par la matrice
$\diag(1,...,p,...,1)$ avec $p$ en la $i$-ième place.
On définit $q_0=q_g=\tr_{F/\QQ}\langle \ ,\ \rangle$,
qui est une forme alternée parfaite sur $\OF^{2g}\otimes\Zp$.
Pour tout schéma $S$ sur $\Zp$, 
on définit la chaîne polarisée standard sur $S$
par le produit tensoriel $St_{\centerdot}\otimes \OS$.
\end{exem}

Nous proposons un analogue de \cite{dJ} proposition 3.6 
sur la forme standard d'une chaîne polarisée arbitraire. 
Cette proposition est démontrée
en toute généralité par  \cite{RZ} théorème 3.16, 
en supposant $p$ localement nilpotent
sur $S$.
Cette dernière condition est enlevée par \cite{P2} theorem 2.2 
et \cite{H} theorem 6.3.

\begin{prop} \label{chainenormalisee}
Soit $M_{\centerdot}$ une chaîne polarisée sur $S$.
Alors localement sur $S$ pour la topologie de Zariski,
$M_{\centerdot}$ est isomorphe à $St_{\centerdot}\otimes\OS$.
De plus, le groupe d'automorphisme de $St_{\centerdot}$ 
est un schéma en groupes lisse sur $\Zp$.
\end{prop}

\subsubsection{}
Nous définissons le modèle local $M_0$
et son diagramme de modèle local.
Soit $(A,\lambda,\iota,\bar{\eta},H_{\centerdot})$ 
un $S$-point de  $Y_0$.
On obtient une chaîne de schémas abéliens sur $S$ avec $\OF$-action
$$
A \to A/H_1 \to \ldots \to A/H_g.
$$
On note $H_0=0$.
Pour tout $i$,
on note $\mathcal{H}^1(A/H_i)$ la première cohomologie de de Rham
de $A/H_i$.
Comme $\lambda: A\to A^{\vee}$
est une $\Zp^{\times}$-polarisation,
elle 
induit un isomorphisme 
$\mathcal{H}^1(A) \xleftarrow{\sim} \mathcal{H}^1(A^{\vee})$
de $\OF\otimes\OS$-module.
L'argument de \cite{dJ} $\S$2 montre que 
pour $j=0$ et $g$,
$\lambda$ définit une forme alternée parfaite
$q_j$ sur $\mathcal{H}^1(A/H_j)$.

\begin{lemm} \label{chaineHun}
La chaîne 
$$
\mathcal{H}^1(A/H_{\centerdot})= 
(\mathcal{H}^1(A) \leftarrow \mathcal{H}^1(A/H_1)
\leftarrow \ldots \leftarrow \mathcal{H}^1(A/H_g),
q_0, q_g
)
$$
est une chaîne polarisée sur $S$.
\end{lemm}

\begin{proof}
On doit vérifier les quatre conditions de 
la définition \ref{chaine}.
La première condition est montrée dans 
\cite{DP} \S 2.8.
La deuxième condition est automatique.
La troisième condition est vérifiée 
par \cite{dJ} proposition 1.7.
Et la dernière condition est montrée par 
\cite{dJ} lemma 2.3 ou \cite{RZ} \S 3.23 d).
\end{proof}

On construit maintenant le diagramme de modèle local.
On définit d'abord un schéma $Z_0$ sur $\Zp$, 
tel que  pour tout schéma $S$ sur $\Zp$,
$Z_0(S)$ est égal à 
l'ensemble des classes d'isomorphismes de
$$
(A,\lambda,\iota,\bar{\eta},H_{\centerdot}, 
\phi)
$$
où $(A,\lambda,\iota,\bar{\eta},H_{\centerdot})$ 
est un $S$-point de $Y_0$,
et 
$$\phi: \mathcal{H}^1(A/H_{\centerdot}) \xrightarrow{\sim} 
St_{\centerdot}\otimes\OS$$ 
est un isomorphisme de
chaînes polarisées.
On note la projection canonique par
$$\pi: Z_0 \to Y_0.$$
La proposition \ref{chainenormalisee} implique que 
$\pi$ est un torseur pour la topologie de Zariski 
sous un schéma en groupes lisse.

\begin{defi} \label{Mzero}
Le modèle local $M_0$ est le schéma sur $\Zp$ 
tel que pour tout schéma $S$ sur $\Zp$,
$M_0(S)$ est égal à 
l'ensemble des classes d'isomorphismes de 
$$
(W_0, \ldots, W_g),
$$
où
\begin{enumerate}[label=--]
\item Pour $i=0,...,g$,
$W_i\subset St_i\otimes\OS$ est un sous-$\OF\otimes\OS$-module,
qui est un $\OS$-facteur direct,
et localement sur $S$ il est isomorphe à $\OF^g\otimes\OS$.

\item La flèche 
$St_{i-1}\otimes\OS \leftarrow St_i\otimes\OS$ 
envoie $W_i$ dans $W_{i-1}$.

\item Pour $j=0$ et $g$, $W_j$ est isotrope pour $q_j$. 

\end{enumerate}
\end{defi}

Par définition $M_0$ est un sous-schéma fermé d'un produit de 
$(g+1)$ Grassmanniennes.
On définit un morphisme 
$$
f: Z_0 \to M_0
$$
tel que, 
soit 
$(A,\lambda,\iota,\bar{\eta},H_{\centerdot},\phi)$
un $S$-point de $Z_0$,
alors son image par $f$ est égale
à $(\phi(\omega_{A}), ..., \phi(\omega_{A/H_g}) )$.
On remarque que 
la propriété d'isotropie 
de $\omega_{A}$ et $\omega_{A/H_g}$
est prouvée par \cite{dJ} \S 2,
donc $f$
est bien défini.
La proposition suivante est un corollaire de la théorie 
de Grothendieck-Messing
sur la déformation des schémas abéliens.

\begin{prop} \label{grothendieckmessing}
Le morphisme $f$ est lisse.
\end{prop}


\section{Schémas de Hilbert-Siegel de niveau $\Gpp$}

\subsection{Théorie de Raynaud}
\subsubsection{}
La construction de schémas de Hilbert-Siegel
en niveau $\Gpp$ est basée sur la théorie de Raynaud
\cite{R},
qui classifie les schémas en $\Fq$-vectoriels de dimension~$1$
vérifiant une certaine hypothèse
sur une base au-dessus de $\Spec D$,
où 
$D$ est un anneau de Dedekind défini dans \cite{R} \S 1.1.
On note $q=p^n$,
$\QQ_q$ l'extension non-ramifiée
de degré $n$ de $\Qp$,
et $\Zq$ l'anneau des entiers de $\QQ_q$.
Nous rappelons la théorie de \cite{R}
en supposant que la base est un schéma sur $\Spec\Zq$.

Dans \cite{R} \S 1.3 (17),
un élément particulier
$$w_p \in p\Zq^{\times}$$ 
est défini.
On note 
$$\chi: \Fq^{\times} \to \Zq^{\times}$$
le caractère de Teichmüller,
qui est un relèvement du morphisme de Frobenius
$$
\Fr: \Fqs\to \Fqs,\quad x\mapsto x^p.
$$
Pour tout $m\in\ZZ$,
on note $\chi^{m}:\Fq^{\times} \to \Zq^{\times}$
la $m$-ième puissance de $\chi$.
On en déduit que tout caractère de $\Fqs$
vers $\Zq^{\times}$ est de la forme
$\chi^{m}$ pour $m\in\{1,...,q-1\}$.

Soit $S$ un schéma au-dessus de $\Spec\Zq$.
Soit 
$H$ un schéma en groupes fini plat de rang $q$ sur $S$ 
muni d'une action de $\Fq$. 
On note $\AAA$ l'algèbre de Hopf de $H$.
Alors $\AAA$ est une $\OS$-algèbre 
localement libre de rang $q$,
et il existe une décomposition
$\AAA=\OS\oplus\JJJ$,
où $\JJJ$ est l'idéal d'augmentation de $H$,
i.e. $\JJJ$ est le faisceau des fonctions sur $H$ 
qui s'annulent en l'élément neutre.
L'action de $\Fq$ sur $H$
induit une action de $\Fqs$ sur $\JJJ$,
donc on peut décomposer $\JJJ$ en composantes $\Fqs$-isotypiques:
$$
\JJJ=\mathop\bigoplus\limits_{m=1}^{q-1} \JJJ_{\chi^m}.
$$
La définition suivante est signalée dans 
\cite{R} \S 1.2, la condition $(\star\star)$.

\begin{defi} \label{groupeRaynaud}
On dit que $H$ est un schéma en groupes  \emph{de Raynaud}
si pour tout $m\in\{1,...,q-1\}$,
le faisceau $\JJJ_{\chi^m}$ 
est un fibré en droites sur $S$
(i.e. il est de rang~$1$).
\end{defi}

Le théorème suivant est le théorème principal de \cite{R}
(\cf théorème 1.4.1)
qui classifie les schémas en groupes de Raynaud.

\begin{theo}[Raynaud] \label{classificationRaynaud}
Soit $S$ un schéma sur $\Spec\Zq$.
Il existe une bijection canonique
entre les classes d'isomorphismes
de schémas en groupes de Raynaud sur $S$
et les classes d'isomorphismes des systèmes
$(\LLL_k, a_k, b_k)_{k=0,...,n-1}$,
où
$\LLL_0, \LLL_1,..., \LLL_{n-1}$
sont des fibrés en droites sur $S$, et
$$
a_k: \LLL_{k-1}^{\otimes p} \to \LLL_{k}, 
\quad b_k: \LLL_{k}\to \LLL_{k-1}^{\otimes p},
\quad k=0,...,n-1
$$
sont des morphismes de fibrés en droites
(avec $\LLL_{-1}=\LLL_{n-1}$)
tels que
$$
a_kb_k=w_p, \quad k=0,...,n-1.
$$
On appelle 
$(\LLL_k, a_k, b_k)_{k=0,...,n-1}$ 
le paramètre du schéma en groupes de Raynaud associé.
\end{theo}

\begin{rema} \label{Raynaudalgebre}
La bijection dans le théorème \ref{classificationRaynaud} est construite de la manière suivante. 
Soit $H/S$ un schéma en groupes de Raynaud,
et $\AAA$ et $\JJJ$ comme précédemment.
Alors le paramètre de $H$ est défini par 
$$\LLL_k=\JJJ_{\cpk},\quad k=0,...,n-1,$$
et $a_k$ et $b_k$
sont 
induits par la multiplication et la comultiplication de $\AAA$,
respectivement.
Réciproquement,
si $(\LLL_k, a_k, b_k)$ est un système 
dans le théorème \ref{classificationRaynaud} ,
alors l'algèbre de Hopf de $H$ est définie par
$$
\mathcal{A}=
\mathop\bigoplus\limits_{0\leq r_k \leq p-1}
(\LLL_0^{\otimes r_0}\otimes
\LLL_1^{\otimes r_1}\otimes
 ... \otimes
\LLL_{n-1}^{\otimes r_{n-1}}).
$$
(Ici on note $\otimes = \otimes_{\OS}$).
Pour tout $\lambda \in \Fqs$,
$\lambda$ agit sur $\LLL_0^{\otimes r_0}\otimes
 ... \otimes
\LLL_{n-1}^{\otimes r_{n-1}}$  
par la multiplication par
$\chi^{r_0+r_1p+...+r_{n-1}p^{n-1}}(\lambda)$,
et $0$ agit par la multiplication par $0$
sur tous les $\LLL_k$.
\end{rema}

\subsubsection{}

Nous définissons la notion de générateur de Raynaud
d'un schéma en groupes de Raynaud.
Soit $S$ un schéma sur $\Spec\Zq$,
et $H/S$ un schéma en groupes de Raynaud.
On note $(\LLL_k, a_k, b_k)$ 
le paramètre de $H$.
On considère le morphisme composé
$$
\LLL_0^{\otimes p^n} 
\xrightarrow{a_1^{\otimes p^{n-1}}} 
\LLL_1^{\otimes p^{n-1}}
\xrightarrow{a_2^{\otimes p^{n-2}}} 
\LLL_2^{\otimes p^{n-2}}
\to
... 
\to
\LLL_{n-2}^{\otimes p^2}
\xrightarrow{a_{n-1}^{\otimes p}} 
\LLL_{n-1}^{\otimes p}
\xrightarrow{a_{0}} \LLL_0;
$$
on note  
$$
a_{0} \cdot a_{n-1}^{\otimes p} \cdot ... \cdot 
a_2^{\otimes p^{n-2}}\cdot a_1^{\otimes p^{n-1}}
\in \Hom_{\OS}(\LLL_0^{\otimes (p^n-1)}, \OS)
$$
l'élément qui correspond à cette flèche composée.

\begin{defi} \label{generateur}
On appelle un \emph{générateur de Raynaud} de $H$
un élément 
$$x \in \Hom_{\OS}(\LLL_0, \OS)$$
qui vérifie l'équation
$$
x^{\otimes (p^n-1)}=  
a_{0} \cdot a_{n-1}^{\otimes p} \cdot ... \cdot 
a_2^{\otimes p^{n-2}} \cdot a_1^{\otimes p^{n-1}}.
$$
\end{defi}

\begin{exem} \label{generateurnaif}
Soit $H$ un schéma en groupes de Raynaud sur $\Spec \CCC$.
Comme tous les fibrés $\LLL_k$ sont triviaux sur $\Spec \CCC$, 
on peut regarder les $a_k, b_k$ comme des éléments de $\CCC$.
Avec les notations précédentes, 
on a 
$$\mathcal{A}= \CCC[t_0, ..., t_{n-1}]/ 
(t_0^p-a_1 t_1,
t_1^p-a_2t_2,...,
t_{n-1}^p-a_{0} t_0).$$
Comme $w_p$ est inversible dans $\CCC$, 
on a $a_k \neq 0$ pour tout $k$.
Donc $H$ est étale sur $\Spec\CCC$,
et $H(\CCC)$ est un espace $\Fq$-vectoriel de dimension $1$.
Un générateur de Raynaud $x$ de $H$ 
définit un point non-nul de $H(\CCC)$
en prenant $t_0=x$,
et vice-versa.
On en déduit que,
dans ce cas la notion de générateur de Raynaud 
et de générateur naïf  coïncident.
\end{exem}

\begin{rema} \label{generateurdual}
Soit $H$ un schéma en groupes de Raynaud sur $S$, et
$(\LLL_k, a_k, b_k)$ son paramètre.
On note $H^{\vee}$ le dual de Cartier de $H$.
On voit facilement que $H^{\vee}$ vérifie la définition \ref{groupeRaynaud}
avec le paramètre $(\LLL_k^{\vee}, b_k, a_k)$.
Supposons 
$$x \in \Hom_{\OS}(\LLL_0, \OS),
\quad y \in \Hom_{\OS}(\LLL_0^{\vee}, \OS)$$
des générateurs de Raynaud de $H$ et $H^{\vee}$, respectivement.
Alors $y$ définit un élément de
$\Hom_{\OS}(\OS, \LLL_{0})$.
On note $xy$ l'élément de $\Gamma(S,\OS)$
qui correspond à la flèche composée
$\OS\xrightarrow{y}\LLL_0\xrightarrow{x}\OS$.
\end{rema}

\subsection{Le schéma $Y_1$}

Nous définissons le schéma de Hilbert-Siegel $Y_1/\Spec\Zq$
de niveau $\Gpp$.
On note $Y_{0,\Zq}$ le changement de base 
du schéma $Y_0$ sur $\Spec \Zq$.
Dans la suite de cet article (sauf \S 5.4),
on fixe un isomorphisme $\OF/p \cong \Fq$.

\begin{lemm} \label{Yzeroplat}
Soit 
$S$ un schéma sur $\Spec\Zq$,
et $(A,\lambda,\iota,\bar{\eta},H_{\centerdot})$
un $S$-point de $\Yq$. 
Alors pour $i=1,...,g$,
$H_i/H_{i-1}$ est un schéma en groupes de Raynaud 
(avec la notation $H_0=0$).
\end{lemm}

\begin{proof}
On voit facilement que la définition \ref{groupeRaynaud}
est préservée par tout changement de base,
donc il suffit vérifier le lemme 
pour le point universel 
$(A,\lambda,\iota,\bar{\eta},H_{\centerdot})$ 
de $\Yq$.
Dans \cite{R} proposition 1.2.2, 
Raynaud a montré que 
si $S$ est un schéma intègre de corps des fractions
de caractéristique $0$,
alors tout schéma en $\Fq$-vectoriels, fini plat de rang $q$
au-dessus de $S$
est un schéma en groupes de Raynaud.
D'après \cite{PZ} theorem 0.1
(qui dépend des résultats de Görtz \cite{G}),
le modèle local $M_0$ est normal.
Comme la normalité est préservée par tout morphisme lisse,
le diagramme de modèle local 
implique que $Y_0$ (donc $\Yq$) est aussi normal.
Donc chaque composante connexe de $\Yq$ est intègre,
et on peut appliquer \cite{R} proposition 1.2.2 pour conclure.
\end{proof}

\begin{defi} \label{Yun}
On définit le schéma $Y_1$ sur $\Spec \Zq$
tel que pour tout $S$ sur $\Spec \Zq$,
$Y_1(S)$ est égal à 
l'ensemble des classes d'isomorphismes de 
$$
(A,\lambda,\iota,\bar{\eta},H_{\centerdot},
x_1,...,x_g,y_1,...,y_g),
$$
où $A,\lambda,\iota,\bar{\eta},H_{\centerdot}$ 
sont comme dans la définition \ref{Yzero},
et pour tout $i$,
$x_i$ (\resp $y_i$) 
est un générateur de Raynaud de 
$H_i/H_{i-1}$ (\resp $(H_i/H_{i-1})^{\vee}$)
tel que
$$x_1y_1=x_2y_2=...=x_gy_g$$
au sens de la remarque \ref{generateurdual}.
\end{defi}

\begin{rema} \label{Yunplat}
Le schéma $Y_1$ est plat sur $\Zq$.
En fait, comme $M_0$ est plat
(\cite{PZ} et \cite{G}),
il suffit de montrer que $Y_1\to \Yq$ est plat.
Mais c'est un corollaire du fait que,
dans la définition \ref{generateur} de générateur de Raynaud
le coefficient dominant de $x$ est égal à $1$.
\end{rema}


\section{Complexes de Lie}

\subsection{Un lemme}

Nous montrons la relation entre les composantes $\Fqs$-isotypiques 
de $[e^*(\III/\III^2) \to e^*i^*\Omega_X]$ dans \S 1
et les paramètres de Raynaud.
On précise d'abord les définitions de $\III$ et $X$.
Soit $S$ un schéma sur $\Spec \Zq$, 
et $H$ un schéma en groupe de Raynaud sur $S$
de paramètre $(\LLL_k, a_k, b_k)$.
On définit deux $\OS$-algèbres 
$$
\SSS=\mathop\bigoplus\limits_{r_k \geq 0}
(\LLL_0^{\otimes r_0}\otimes
\LLL_1^{\otimes r_1}\otimes
 ... \otimes
\LLL_{n-1}^{\otimes r_{n-1}})
$$
$$
\AAA=
\mathop\bigoplus\limits_{0\leq r_k \leq p-1}
(\LLL_0^{\otimes r_0}\otimes
\LLL_1^{\otimes r_1}\otimes
 ... \otimes
\LLL_{n-1}^{\otimes r_{n-1}}),
$$
donc
$H=\underline{\Spec}_{\OS}(\AAA)$.
On définit $X=\underline{\Spec}_{\OS}(\SSS)$. 
Les morphismes $a_k: \LLL_{k-1}^{\otimes p}\to \LLL_k$
induisent  une surjection $\SSS\twoheadrightarrow \AAA$,
et on note $\III$ son noyau.
L'action 
de  $\Fqs$ sur $\AAA$ (\cf la remarque \ref{Raynaudalgebre})
s'étend à une action sur $\SSS$ 
telle que les $a_k$ et $\SSS \twoheadrightarrow \AAA$ 
sont $\Fqs$-équivariants.
On en déduit que  l'immersion fermée $i: H\to X$
définie par $\III$ est $\Fqs$-équivariante.

On note $e: S\to H$ la section unité, et
$\dd: \III/\III^2 \to i^*\Omega_{X/S}$
la différentielle.
Pour $k=0,...,n-1$,
on note 
$$
[e^*(\III/\III^2) \xrightarrow{\dd} e^*i^*\Omega_{X/S}]_{\chi^{p^k}}
$$
la composante 
$\chi^{p^k}$-isotypique
de
$[e^*(\III/\III^2) \xrightarrow{\dd} e^*i^*\Omega_{X/S}]$.
(On rappelle que $\chi: \Fqs \to \Zqs$ est le caractère de Techmüller.)

\begin{lemm} \label{lemmedepart}
Il existe un isomorphisme canonique de complexe
$$
[\LLL_{k-1}^{\otimes p} \xrightarrow{a_k} \LLL_k ]
\cong
[e^*(\III/\III^2) \xrightarrow{\dd} e^*i^*\Omega_{X/S}]_{\chi^{p^k}}
$$
pour $k=0,...,n-1$.
\end{lemm}

\begin{proof}

On note $\NNNNN$ le noyau de $e^*i^*:\SSS\to\OS$
la projection
sur le premier facteur de $\SSS$.
Alors 
$e^*(\III/\III^2) =
e^*i^*(\III)
=\III/\III\NNNNN.
$
Pour tout $k$, on définit un morphisme de $\OS$-modules
$$
\varphi_k: \LLL_{k-1}^{\otimes p} \to \SSS
=\mathop\bigoplus\limits_{r_j \geq 0}
(\LLL_0^{\otimes r_0}\otimes
\LLL_1^{\otimes r_1}\otimes
 ... \otimes
\LLL_{n-1}^{\otimes r_{n-1}})
$$
qui envoie un élément $u$ de $\LLL_{k-1}^{\otimes p}$
à $(-u, a_ku)$ de la composante
$\LLL_{k-1}^{\otimes p}\oplus \LLL_k$
de $\SSS$.
Par définition,
$\varphi_k(\LLL_{k-1}^{\otimes p})\subset \III$.
On note
$$
\varphi:
\LLL_{n-1}^{\otimes p}\oplus
\LLL_0^{\otimes p}
\oplus ... \oplus 
\LLL_{n-2}^{\otimes p}
\to 
\III/\III\NNNNN
$$
la somme de tous les $\varphi_k$.

On montre que $\varphi$
est un isomorphisme.
Comme la question est d'une nature locale,
on peut supposer que tous les $\LLL_k$ 
sont triviaux sur $S$,
et identifier les $a_k$ avec des éléments de $\Gamma(S,\OS)$.
On obtient des isomorphismes naturels
$$
\SSS=\OS[x_0,...,x_{n-1}]
$$
$$
\AAA=\OS[x_0,...,x_{n-1}]/
(x_0^p-a_1x_1,x_1^p-a_2x_2,...,x_{n-1}^p-a_0x_0)
$$
$$
\III=
(x_0^p-a_1x_1,x_1^p-a_2x_2,...,x_{n-1}^p-a_0x_0)
$$
$$
\NNNNN=(x_0,x_1,...,x_{n-1})
$$
$$
\varphi_k: \OS \to \OS[x_0,...,x_{n-1}],
\quad 1\mapsto a_kx_k-x_{k-1}^p.
$$
Par la définition de $\varphi$,
on voit directement qu'il est surjectif.
L'injectivité de $\varphi$ vient d'une comparaison 
de degré des éléments dans $\III$ et $\III\NNNNN$.

Comme $\SSS=\uSym_{\OS}(\LLL_0\oplus\LLL_1\oplus...\oplus\LLL_{n-1})$ ,
il existe un isomorphisme canonique
$$
\psi: \LLL_0\oplus\LLL_1\oplus...\oplus\LLL_{n-1}
\xrightarrow{\sim} 
\Omega_{X/S}\otimes_{\SSS}\OS =e^*i^*\Omega_{X/S}.
$$
On obtient un diagramme 
$$
\xymatrix{
\oplus\LLL_{k-1}^{\otimes p}
\ar[r]^{\oplus a_k}
\ar[d]^{\varphi}_{\rotatebox{270}{$\sim$}}&
\oplus\LLL_{k}
\ar[d]^{\psi}_{\rotatebox{270}{$\sim$}}
\\
e^*(\III/\III^2)
\ar[r]^{\dd}
&e^*i^*\Omega_{X/S}
}.
$$
En prenant des coordonnées locales comme précédemment,
on voit que ce diagramme est commutatif.
Par définition $\varphi$ et $\psi$ 
sont $\Fqs$-équivariants.
On obtient l'isomorphisme dans l'énoncé quand on prend
les composantes $\Fqs$-isotypiques
(on rappelle que $\Fqs$ agit sur $\LLL_k$ via $\cpk$).
\end{proof}

\begin{rema} \label{caracterefondamental}
Dans la démonstration on voit que 
le complexe $[e^*(\III/\III^2) \xrightarrow{} e^*i^*\Omega_{X/S}]$
ne possède que des composantes isotypiques 
pour les caractères $\cpk$,
i.e. les caractères \emph{fondamentaux}.
Ce phénomène 
est en fait un corollaire du complexe
$_{\Fq}{\ell}_H^{\vee}$
de notre redéfinition.
\end{rema}

\subsection{Les complexes équivariants d'Illusie}
\subsubsection{}

On rappelle les deux complexes équivariants définis par Illusie \cite{I}.
On suppose que $S$ est un schéma.
On note 
$\Sfpqc$ le grand site fpqc de $S$,
et $\OO_{\Sfpqc}$ son faisceau structurel.

Soit $M$ un monoïde commutatif,
et $K/S$ un schéma en groupes plat et de présentation finie,
muni d'une action de $M$.
On note l'anneau $\ZZ[M]=\mathop\bigoplus\limits_{m\in M}\ZZ m$,
avec la multiplication et l'unité 
définies par celles de $M$.
La théorie de \cite{I} VII 2.2 définit le 
complexe de co-Lie $M$-équivariant de $K$:
$$
\llMK \in \Db(\ZZ[M]\otimes \OO_{\Sfpqc}).
$$
On remarque que le complexe défini dans \cite{I} est 
le complexe cotangent équivariant $\LL_K^M$;
par définition $\llMK$ est égal à $\LL e^*\LL_K^M$
avec $e$ la section unité de $K$.

Soit d'autre part
$A$ un anneau commutatif,
et $G$ un schéma en $A$-modules,
plat et de présentation finie 
sur $S$.
Dans \cite{I} VII 4.1 Illusie a défini 
le complexe de Lie $A$-équivariant de $G$:
$$\AllGv \in \Db(A\otL\OO_{\Sfpqc}).$$
Ici $\Db(A\otL\OO_{\Sfpqc})$ est la catégorie dérivée 
du dg-anneau $A\otL\OO_{\Sfpqc}$, \cf \S 4.3.

\subsubsection{}

On explique les idées dans les définitions de $\llMK$ et $\AllGv$.
Commençons par $\llMK$.
Un objet basique est le \emph{nerf} $\Ner(M,K)$.
C'est le schéma simplicial 
$$
\xymatrix{
\cdots 
\ar[r]
\ar@<-1ex>[r]
\ar@<1ex>[r]
\ar@<2ex>[r]
&M\times M\times K
\ar[r]_{\quad \dd_2}
\ar@<1ex>[r]|-{\dd_1}
\ar@<2ex>[r]^{\quad \dd_0}
&M\times K
\ar[r]|-{\dd_1}
\ar@<1ex>[r]^{\quad \dd_0}
\ar@<2ex>@/^/[l]|-{\sss_0}
\ar@<3ex>@/^/[l]^-{\sss_1}
&K
\ar@<1ex>@/^/[l]^{\quad \sss_0}
},
$$
où les flèches $\dd_i$ et $\sss_j$ 
sont définis dans \S 4.5.2:
par exemple entre les deux derniers sommets,
$\dd_0$ est l'action de $M$ sur $K$,
$\dd_1$ est la projection sur $K$,
et $\sss_0$ est la section neutre de $M$.
On obtient le schéma simplicial 
$\Ner(M,S)$ par la même définition (en remplaçant $K$ par $S$),
tel que  $\Ner(M,K)$ est un schéma simplicial au-dessus de  $\Ner(M,S)$
par un morphisme structurel.

On considère le complexe cotangent $\LL_{\Ner(M,K)/\Ner(M,S)}$
associé à ce morphisme structurel.
L'idée est qu'on utilise $\LL_{\Ner(M,K)/\Ner(M,S)}$
pour définir  une $M$-structure sur le complexe cotangent usuel
$\LL_{K/S}$.
Comme on travaille sur les complexes de co-Lie,
on utilise plutôt 
$\ell_{\Ner(M,K)}=\LL e^* \LL_{\Ner(M,K)/\Ner(M,S)}$ ,
où $e: \Ner(M,S)\to \Ner(M,K)$ est la section unité.

On note $\underline{\ell}_{\Ner(M,K)}$ le complexe 
sur le grand site fpqc associé à $\ell_{\Ner(M,K)}$.
C'est un objet de $\Db(\Ner(M,S)_{\fpqc})$,
où $\Ner(M,S)_{\fpqc}$ est le \emph{site total} fpqc de $\Ner(M,S)$.
On rappelle la définition du site et topos total de \cite{I} VI 5.1:
un faisceau sur $\Ner(M,S)_{\fpqc}$
est un système $(N_n)_{n\in\NNN}$,
où $N_n$ est un faisceau fpqc sur $M^n\times S$,
avec les morphismes de transition
$\dd_i^*N_n\to N_{n+1}$ et $\sss_j^* N_{n+1}\to N_n$
pour tout $n,i$ et $j$,
vérifiant les conditions de cocycle.

Alors on a un foncteur naturel
$$
\ner: \Db(\ZZ[M]\otimes \OO_{\Sfpqc})
\to
\Db(\Ner(M,S)_{\fpqc})
$$
induit par l'action de $M$,
\cf la définition \ref{nerd} pour la définition précise.
On veut définir $\llMK$ via l'isomorphisme
$$
\ner(\llMK)\cong \underline{\ell}_{\Ner(M,K)};
$$
mais on doit montrer  que 
$\underline{\ell}_{\Ner(M,K)}$ est dans l'image essentielle 
de $\ner$.
C'est démontré par un théorème clé de \cite{I}
(\cf le théorème \ref{thmillusie})
avec la pleine fidélité de $\ner$.
On obtient donc la définition de $\llMK$.

\subsubsection{}

On considère maintenant le complexe $\AllGv$.
La situation est plus compliquée parce que 
$A$ contient plus de structure qu'un monoïde.
Pour rappeler les idées de \cite{I},
on va utiliser la langage des diagrammes
(\cf \S 4.4.1).
On note pour toute catégorie $T$ et tout $r\in\NNN$,
$\Diag_r(T)$ (\resp $\Simp_r(T)$) la catégorie des $r$-diagrammes
(\resp objets $r$-simpliciaux) de $T$.

Supposons que $T$ est une catégorie abélienne.
Un outil clé de \cite{I} est un foncteur particulier 
$$\CCCC: T\to \Simp_1(\Diag_1(T)),$$ 
voir \S 4.4 pour la définition précise.
On note $\Mod_{\ZZ}$ la catégorie des $\ZZ$--modules.
Alors $\CCd A(1)$ est un objet de monoïde 
dans $\Simp_1(\Diag_1(\Mod_{\ZZ}))$,
et on a une action de $\CCd A(1)$ sur $\CCd G(1)$.
On peut imiter les constructions de \S 4.2.2 
pour $M=\CCd A(1)$ et $K=\CCd G(1)$:
on définit un foncteur (\cf la définition \ref{nerd})
$$
\nerd: 
\Db(\ZZ[\CC_{\centerdot}A(1)]\otimes \OO_{\Sfpqc})
\to
\Db(\Ner(\CC_{\centerdot}A(1), \CC_{\centerdot}G(1))_{\fpqc});
$$
on considère le complexe de Lie 
$\underline{\ell}^{\vee}_{\Ner(\CC_{\centerdot}A(1), \CC_{\centerdot}G(1))}$,
et on applique le théorème \ref{thmillusie}.
On déduit qu'il existe
un unique objet $L \in \Db(\ZZ[\CC_{\centerdot}A(1)]\otimes \OO_{\Sfpqc})$,
tel que
$$
\nerd(L)\cong
\underline{\ell}^{\vee}_{\Ner(\CC_{\centerdot}A(1), \CC_{\centerdot}G(1))}.
$$

Les propriétés particulières de $\CCCC$ permettent de descendre le complexe $L$
à la catégorie $\Db(A\otL\OO_{\Sfpqc})$.
Pour le décrire, on utilise la notation $\CCd\SSSS_1$ de \S 4.4,
qui est un $1$-diagramme de catégories.
On note $\Diag_{\CCd\SSSS_1}(T)$ 
la catégorie des $2$-diagrammes 
de type $\CCd\SSSS_1$ dans $T$.
Par définition on sait que les images de $\CCCC$ appartiennent à  $\Diag_{\CCd\SSSS_1}(T)$.
Par \cite{I} VI 9.5 et VI 11.5.2.2,
il existe un foncteur 
$$\Delred: \Diag_{\CCd\SSSS_1}(\Mod_{\ZZ})
\to \Simp_1(\Mod_{\ZZ}),$$
tel que son composé avec $\CCCC$
est quasi-isomorphe au foncteur 
de l'objet simplicial constant
de $\Mod_{\ZZ}$ à $\Simp_1(\Mod_{\ZZ})$.
Dans \cite{I} VI 9.5,
Illusie considère le foncteur 
$$
\ZZ^{\st}[-]= 
\Delta^2(\CCd \ZZ[-](1))^{\red}(-1)
: \Mod_{\ZZ}\to \Simp_1(\Mod_{\ZZ}).
$$
On note $\NN\ZZ^{\st}[-]$ le complexe de $\ZZ$-modules
associé à $\ZZ^{\st}[-]$ via la correspondance de Dold-Kan,
et $\tau_{\geq -1}\NN\ZZ^{\st}[-]$ sa troncation 
en degré $-1$.
Une propriété clé est que $\tau_{\geq -1}\NN\ZZ^{\st}[-]$
est quasi-isomorphe au foncteur identité de $\Mod_{\ZZ}$,
\cf \cite{I} VI 11.4.1.5.

Ces propriétés de $\CCCC$ définissent
un diagramme commutatif (\cf \cite{I} VI 11.5.2.2)
$$
\xymatrix{
\DD(A\otL\OO_{\Sfpqc})
\ar[r]^-{\CCCC}
\ar[rd]
&\DD(\ZZ[\CC_{\centerdot}A(1)]\otimes \OO_{\Sfpqc})
\ar[d]^{\Delred}\\
&\DD(\NN\ZZ^{\st}[A]\otimes \OO_{\Sfpqc})
},
$$
où la  flèche oblique est la restriction de scalaire
via $\NN\ZZ^{\st}[A]\to A$.
Comme $\tau_{\geq -1}\NN\ZZ^{\st}[A]$ est une résolution de $A$,
on a une équivalence de catégorie
$$
\DD^{[0,1]}(A\otL\OO_{\Sfpqc})
\xrightarrow{\sim}
\DD^{[0,1]}(\NN\ZZ^{\st}[A]\otimes \OO_{\Sfpqc}).
$$
Cette équivalence nous permet de descendre $L$ 
de $\DD^{[0,1]}(\ZZ[\CC_{\centerdot}A(1)]\otimes \OO_{\Sfpqc})$ à 
$\DD^{[0,1]}(A\otL\OO_{\Sfpqc})$.
On définit donc $\AllGv\in \DD^{[0,1]}(A\otL\OO_{\Sfpqc})$ par l'isomorphisme
$$
\CC_{\centerdot}{\AllGv}(1)
\cong L.
$$

\subsubsection{}

On explique pourquoi la généralisation de $\AllGv$
au petit site de Zariski est possible.
On voit que dans \S 4.2.3, la définition consiste en deux étapes:
une descente par $\nerd$ et une descente par $\CCCC$.
La partie pour $\CCCC$ est purement algébrique,
donc se généralise automatiquement.
Néanmoins la généralisation sur $\nerd$ n'est pas évidente:
avec le petit site de Zariski,
on ne voit pas pourquoi on peut appliquer
le théorème \ref{thmillusie}.
On le verra via un autre site (qu'on appelle le site des \emph{isomorphismes locaux}),
tel que son topos est équivalent au topos de Zariski.
Cette idée sera expliquée dans \S 4.6 en détail.

\subsection{Les dg-anneaux}

Dans cette section on rappelle les rudiments 
sur les dg-anneaux 
(anneaux différentiels gradués)
utilisés dans \cite{I}.
On ne considère que des 
dg-anneaux de degré $\leq 0$.
Par définition,
un \emph{dg-anneau} $P_{\centerdot}$ est un complexe de $\ZZ$--modules
$$
...\to P_2 \to P_1 \to P_0
$$
muni d'une multiplication
$$
P_{\centerdot} \otimes P_{\centerdot} \to P_{\centerdot}
$$
et un élément unité dans $P_0$,
tel que les axiomes usuels d'anneau
sont vérifiés.
(Ici $P_{\centerdot} \otimes P_{\centerdot}$ est le complexe total
d'un double complexe.)
Un $P_{\centerdot}$--\emph{module} $M_{\centerdot}$
est un complexe de $\ZZ$--modules
muni d'une action de $P_{\centerdot}$.
On note $\CC(P_{\centerdot})$ la catégorie des $P_{\centerdot}$--modules.
La catégorie dérivée
$\DD(P_{\centerdot})$ est définie comme
la localisation de $\CC(P_{\centerdot})$
par l'ensemble des quasi-isomorphismes.
La même définition vaut pour 
$\DD^+(P_{\centerdot}),\DD^-(P_{\centerdot})$ et $\Db(P_{\centerdot})$.

Soient $A$ et $B$ deux anneaux commutatifs usuels.
Alors la définition de 
$$A\otL B$$ 
dépend du choix 
d'une résolution de $A$ (ou de $B$).
Soit $P_{\centerdot}$ est une résolution plate de $A$
qui est muni d'une structure de dg-anneau,
tel que $P_{\centerdot} \to A$
est un morphisme de dg-anneau.
Alors le complexe $P_{\centerdot} \otimes B$
est aussi un dg-anneau.
On peut définir
$$
\DD(A\otL B):=\DD(P_{\centerdot} \otimes B).
$$
La catégorie $\DD(A\otL B)$ 
est bien définie à équivalence près,
\cf \cite{I} VI 10.3.15.
Comme $\ZZ$ est de tor-dimension $1$,
on peut prendre $P_{\centerdot}$ d'une forme plus simple:
$$
P_{\centerdot}= [P_1\to P_0]
$$
où $P_0$ est un anneau plat sur $\ZZ$,
la flèche est injective
et l'image de $P_1$ est un idéal de $P_0$.

Soit maintenant $A$ un anneau commutatif
et $S$ un schéma.
On note $\OS$ le faisceau structurel du schéma $S$.

\begin{lemm} \label{ALS}
Si $S$ est plat sur $\Spec\ZZ$,
alors $A\otimes \OS$
est un représentant canonique de $\ALS$.
\end{lemm}

\begin{proof}
Soit $[P_1\to P_0]$ une résolution  de $A$
de la forme précédente.
Il suffit de voir que 
$$
0\to 
P_1\otimes\OS
\to P_0 \otimes\OS
\to A\otimes\OS
\to 0
$$
est exacte
en prenant la localisation en chaque point de $S$.
\end{proof}

\begin{coro} \label{DALS}
Si $S$ est plat sur $\Spec\ZZ$,
alors il existe  
une équivalence canonique 
$\DD^*(\ALS) \cong  \DD^*(A\otimes \OS)$
pour $*=\varnothing, +, -, b$.
\end{coro}

\subsection{Le foncteur $\CCCC$}
\subsubsection{}
Dans cette section on rappelle la définition 
du foncteur $\CCCC$ de \cite{I}.
Ce foncteur est une composition de deux foncteurs
$\text{--}(1)$ et $\CCd\text{--}$.
On va définir la translation $\text{--}(1)$ dans \S 4.4.2,
puis le foncteur $\CCd\text{--}$ dans \S 4.4.3.

On fixe d'abord les notations 
pour les objets multi-simpliciaux
et les multi-diagrammes.
On note 
$\ddd$ la catégorie des simplexes standards.
Par définition les objets de $\ddd$ sont les ensembles finis
$[n]:=\{0,1,...,n\}$
pour tout $n\in\NNN$.
Les morphismes de $\ddd$ sont engendrés 
par les flèches
$$
\dd^i: [n]\to [n+1],\quad i=0,...,n+1
$$
$$
\sss^j: [n+1]\to [n],\quad j=0,...,n
$$
où $\dd^i$ est le morphisme croissant en oubliant $i$,
et $\sss^j$ est le morphisme croissant en répétant $j$.
On note $\ddo$ la catégorie opposée de $\ddd$.

On note $\ddb$ la catégorie des simplexes standards augmentés.
Par définition les objets de $\ddb$ 
sont tous les objets de $\ddd$
avec un objet supplémentaire:
l'ensemble vide $\varnothing$.
Les morphismes de $\ddb$ sont engendrés par 
les $\dd^i, \sss^j$ 
et l'unique morphisme $\varnothing \to [0]$.
On note $\ddp$ la catégorie des simplexes stricts augmentés
(\cite{I} VI 4.1).
C'est la sous-catégorie de $\ddb$
en gardant les mêmes objets
mais
en oubliant tous les morphismes \emph{non-injectifs}
(i.e. les morphismes composés par au moins un $\sss^j$).

Soit $T$ une catégorie, 
et $r\in\NNN$.
Par définition, un objet $r$-\emph{simplicial} de $T$
est un foncteur de $(\ddo)^r$ dans $T$.
On note $\Simp_r(T)$ la catégorie des objets $r$-simpliciaux de $T$.

On utilise la langage des multi-diagrammes
(\cite{I} VI 5.6 b)).
Par définition,
un $1$-\emph{diagramme} de $T$  est un foncteur 
d'une petite catégorie dans $T$.
Soit $X: I\to T$ un tel foncteur.
La catégorie $I$ est appelée le \emph{type} de $X$.
Pour deux $1$-diagrammes $X: I \to T$ et $Y: J\to T$,
un morphisme de $X$ vers $Y$ consiste en un couple 
$(u,v)$,
où $u$ est un foncteur de $I$ vers $J$,
et $v$ est un morphisme de foncteur entre 
$X$ et $Y\circ u$.
On note la catégorie des $1$-diagrammes de $T$ par $\Diag_1(T)$.
Par récurrence,
on peut définir pour tout $r\geq 1$ la catégorie 
des $r$-diagrammes de $T$ 
en prenant $\Diag_r(T)=\Diag_1(\Diag_{r-1}(T))$.
Le foncteur de type s'étend en un foncteur 
de $\Diag_r(T)$ à $\Diag_{r-1}(\Cat)$,
où $(\Cat)$ est la catégorie des petites catégories.

\subsubsection{}
Supposons que $T$ est une catégorie abélienne.
On rappelle la définition du foncteur 
$\text{--}(1): T\to \Diag_2(T)$ 
de \cite{I}.
On définit d'abord un $1$-diagramme 
$$\SSSS_1: \ddp\to(\Cat)$$
tel que pour tout $n\in\Nm$
(avec $[-1]:=\varnothing$),
$$
\SSSS_1([n])={(\ddo)}^{n+1},
$$
où  ${(\ddo)}^{0}=\pt$
est la catégorie pointée.
Pour tout $n$ 
et $i=0,...,n+1$,
le morphisme
$\dd^i: \SSSS_1([n])\to \SSSS_1([n+1])$
est défini par la formule
$$\dd^i(x_0,...,x_n)=(x_0,...,x_{i-1}, [1], x_i,...,x_n).$$
On voit que les $\dd^i$ vérifient bien
la loi de composition de $\ddp$,
donc $\SSSS_1$ est un foncteur.

On a une description concrète pour un  $2$-diagramme de type $\SSSS_1$.
Soit $X$ est un tel $2$-diagramme dans $T$.
Alors la donnée de $X$
est équivalente au système
$$
(X^r, \dd^{r,i})_{r\in\NNN, 0\leq i \leq r}
$$
où pour tout $r$ et tout $i$,
$X^r \in \Simp_r(T)$
et 
$\dd^{r,i}: X^r \to \partial^1_i X^{r+1}$ 
un morphisme de $\Simp_r(T)$
dont la loi de composition 
est déterminée par celle de $\ddp$. 
Ici $\partial^1_i$ est le foncteur de la 
$i$-ième face de degré $1$:
il est égal à la composition
avec $\dd^i: (\ddo)^r\to (\ddo)^{r+1}$.
(On peut prendre $X^r=X([r-1])$
pour tout $r\in\NNN$.)

Pour tout $r$, on a le foncteur de translation
des objets $r$-simpliciaux
$$
[1,...,1]: \Simp_r(T) \to \Simp_r(T).
$$
Via la correspondance de Dold-Kan,
il devient la translation usuelle des $r$-complexes.
La définition suivante apparaît dans \cite{I} VI 9.2.

\begin{defi} \label{translation}
On définit le foncteur
$$
\text{--}(1): T \to \Diag_2(T)
$$
tel que pour tout objet $X$ de $T$,
$X(1)$ est le $2$-diagramme de $T$ de type $\SSSS_1$
qui vérifie 
$$X(1)^r
=
X[1,...,1]$$ 
pour tout $r\in\NNN$, 
où le $X$ à droite est l'objet $r$-simplicial constant 
de valeur $X$;
pour $i=0,...,r$, on définit
$$
\dd^{r,i}=\id_{X(1)^r}.
$$
\end{defi}
\noindent
Voici un dessin de $X(1)$,
où on trouve $X(1)^r$ avec $r=0,1,2$.
$$
\xymatrix{
&X\oplus X
\ar@{.>}[r]^{\id}
\ar[d]
\ar@<1ex>[d]
\ar@<-1ex>[d]
&X\oplus X
\ar[r]
\ar@<1ex>[r]
\ar[d]
\ar@<1ex>[d]
\ar@<-1ex>[d]
&0
\ar[d]
\ar@<1ex>[d]
\ar@<-1ex>[d]\\
&X
\ar@{.>}[r]^{\id}
\ar[d]
\ar@<1ex>[d]
&X
\ar[d]
\ar@<1ex>[d]
\ar[r]
\ar@<1ex>[r]
&0
\ar[d]
\ar@<1ex>[d]\\
X
\ar@{.>}[ru]^{\id}
&0
\ar@{.>}[r]^{\id}
&0
\ar[r]
\ar@<1ex>[r]
&0
}
$$

\subsubsection{}

On rappelle la définition du foncteur $\CCd\text{--}$
dans \cite{I} VI 11.1.2.
Soit $T'$ une catégorie 
qui possède des sommes quelconques.
On décrit le foncteur
$$
\CCd\text{--}: \Diag_1(T')\to \Simp_1(T').
$$
Soit $U: I\to T'$ un $1$-diagramme de $T'$.
Alors $\CCd U$ est un objet simplicial de $T'$,
dont
le $n$-ième étage est 
$$
\CC_nU=
\mathop\bigsqcup\limits_
{r_0\to...\to r_n} U_{r_0}
$$
pour tout $n\in\NNN$,
où $r_0\to...\to r_n$ parcourt toutes les chaînes de longueur $n$
de flèches dans $I$.
Pour $i=0,...,n+1$, le morphisme
$\dd_i: \CC_{n+1}U\to \CC_nU$
envoie le facteur indexé par 
$r_0\to...\to r_{n+1}$
vers celui indexé par 
$r_0\to...\to r_{i-1}\to r_{i+1}\to ...\to r_{n+1}$
en composant les flèches $r_{i-1}\to r_i\to r_{i+1}$.
Sur le facteur $U_{r_0}$,
$\dd_0$ est le morphisme fonctoriel
$U_{r_0}\to U_{r_1}$,
et $\dd_i$ est l' identité si $i\geq 1$.
On a une description similaire pour les morphismes $\sss_j$
en répétant $r_j$ dans l'index.

Soit $T$ une catégorie abélienne.
On peut donc définir le foncteur 
$$\CCCC: T\to \Simp_1(\Diag_1 T)$$
comme la composition de 
$\text{--}(1)$ et $\CCd\text{--}$,
i.e. dans la définition de $\CCd\text{--}$,
on prend $U$ égal à 
$$X(1): \ddp\to \Diag_1(T),\quad
[n]\mapsto X(1)^{n+1}$$
pour un objet $X$ de $T$.
On obtient pour tout $n\in\NNN$,
$$
\CC_nX(1)=\mathop\bigsqcup\limits_
{[m_0]\to...\to [m_n]}
X(1)^{m_0+1}
$$
où $[m_0]\to...\to [m_n]$ parcourt toutes les chaînes 
de longueur $n$ des flèches de $\ddp$. 
\footnote{La somme $X\sqcup Y$ de deux objets $X: I\to T$ et $Y: J\to T$
de $\Diag_1(T)$
est l'extension évidente de $X$ et $Y$ à $I\sqcup J$,
la réunion disjointe des catégories $I$ et $J$.}
On remarque que $\CCd X(1)$ est un $2$-diagramme 
de type $\CCd\SSSS_1$.

\subsubsection{}
On note $\Mod_{\ZZ}$ la catégorie des $\ZZ$--modules.
Soit $A$ un anneau commutatif,
et on le voit comme un objet de  $\Mod_{\ZZ}$.
Les constructions de \S 4.4.2, \S 4.4.3 nous donnent les $2$-diagrammes 
$A(1)$ et $\CCd A(1)$ de $\Modz$.
La structure d'anneau sur $A$ définit les structures de monoïde
sur $A(1)$ et $\CCd A(1)$ 
qu'on rappelle ici
(\cf \cite{I} VI 11.1).

Le  $2$-diagramme $A(1)$ est 
de type $\SSSS_1: \ddp \to (\Cat).$
Il existe une structure de monoïde évidente 
sur la catégorie $\ddp$, telle que
la multiplication est définie par le foncteur
$$
\ddp\times\ddp\to\ddp,\quad ([m],[n])\mapsto [m+n+1],
$$
et l'élément neutre est l'objet $[-1]$.
Il existe une structure de monoïde 
sur $\SSSS_1$ au-dessus de $\ddp$:
la multiplication $\SSSS_1\times\SSSS_1\to\SSSS_1$
est les identifications
$(\ddo)^{m+1}\times (\ddo)^{n+1} =(\ddo)^{m+n+2}$.
Pour définir la structure de monoïde sur $A(1)$,
il suffit de définir une multiplication 
$A(1)\times A(1) \to A(1)$
dans $\Diag_2(\Modz)$,
qui est équivalente aux morphismes d'objets $(r+s)$-simpliciaux
$$
A(1)^r\times A(1)^s \to A(1)^{r+s}
$$
pour tout $r,s\in\NNN$.
Mais d'après la définition \ref{translation} 
on voit qu'un tel morphisme est
induit par la \emph{multiplication} de $A$
de façon évidente.
Par la fonctorialité de $\CCCC$,
la multiplication de $A(1)$ 
définit une multiplication
sur $\CCd A(1)$,
tel que $\CCd A(1)$ est un monoïde dans $\Simp_1(\Diag_1 \Modz)$.

On note $\ZZ[\CCd A(1)]$ l'objet de $\Simp_1(\Diag_1\Modz)$
obtenu en appliquant $\ZZ[\text{--}]$
sur chaque sommet de $\CCd A(1)$.
La multiplication de $\CCd A(1)$ induit 
un morphisme
$$
\ZZ[\CCd A(1)] \otimes \ZZ[\CCd A(1)]\to \ZZ[\CCd A(1)],
$$
tel que $\ZZ[\CCd A(1)]$ est muni d'une structure d'anneau
au sens de \cite{I} VI 11.2.2.
On peut considérer les $\ZZ[\CCd A(1)]$-modules:
par définition,
un $\ZZ[\CCd A(1)]$-\emph{module}
est un objet de $\Simp_1(\Diag_1\Modz)$ de type $\CCd\SSSS_1$
muni d'une action de $\ZZ[\CCd A(1)]$ 
par $\otimes$ comme plus haut.

Soit $S$ un schéma. 
On note $\Mod(\OS)$ la catégorie des faisceaux de $\OS$-modules.
On note $\Mod(\CCd\SSSS_1,\OS)$
la catégorie des
$2$-diagrammes de $\Mod(\OS)$ 
de type $\CCd\SSSS_1$,
et
$\Mod(\ZZ[\CCd A(1)]\otimes\OS)$
la catégorie des $\ZZ[\CCd A(1)]\otimes\OS$-modules.
On a un foncteur d'oubli naturel de 
$\Mod(\ZZ[\CCd A(1)]\otimes\OS)$ 
dans $\Mod(\CCd\SSSS_1,\OS)$.
Le foncteur $\CCCC$ induit 
un foncteur exact
$$
\CCCC:
\Mod(A\otimes\OS)
\to
\Mod(\ZZ[\CCd A(1)]\otimes\OS),\quad
V\mapsto\CCd V(1).
$$
La proposition suivante de \cite{I} est la propriété clé pour le foncteur
$\CCCC$,
où $\DD^*(...)$ est la catégorie dérivée de $\Mod(...)$.

\begin{prop} \label{Cun}
Supposons que $S$ est plat sur $\Spec\ZZ$.
Alors le foncteur 
$$
\CCCC:
\DD^{[0,1]}(A\otimes\OS)
\to
\Db(\ZZ[\CCd A(1)]\otimes\OS)
$$
est pleinement fidèle,
et son image essentielle se compose des objets $Y\in\Db(\ZZ[\CCd A(1)]\otimes\OS)$ tels que,
en tant qu'objet de $\Db(\CCd\SSSS_1,\OS)$,
$Y$ est dans l'image essentielle de 
$
\CCCC:
\DD^{[0,1]}(\OS)
\to
\Db(\CCd\SSSS_1,\OS)
$.
\end{prop}

\begin{proof}
Via le corollaire \ref{DALS},
c'est un cas particulier de \cite{I} VI 11.5.2.4
en prenant $\OO=\OS$.
\end{proof}

\subsection{Le foncteur $\nerd$}
\subsubsection{}
Soit $S$ un schéma et 
$A$ un anneau commutatif.
Par la construction de \S 4.4,
on a les $2$-diagrammes $\CCd A(1)$ et $\CCd S(1)$
de type $\CCd\SSSS_1$.
Dans cette section
on définit une action
de $\CCd A(1)$ sur $\CCd S(1)$
induite par l'action triviale de $A$ sur $S$.
On considère dans \S 4.5.2 et \S 4.5.4 le diagramme $\Ner$ et le foncteur $\nerd$ 
associé à cette action.

On définit d'abord $\CCd S(1)$.
On applique la définition \ref{translation} lorsque $T$ est 
la catégorie des faisceaux abéliens fpqc sur $S$,
et $X=S$.
On obtient le $2$-diagramme $S(1)$ de $T$,
tel que  pour tout $r\in\NNN$, 
$S(1)^r$ est le schéma $r$-simplicial constant de valeur $S$.
On note $(\Sch/S)$
la catégorie des schémas sur $S$,
et on voit $S(1)$ comme un $2$-diagramme dans $(\Sch/S)$.
En appliquant $\CCCC$
on obtient le $2$-diagramme $\CCd S(1)$.
Par la fonctorialité,
l'action triviale de $A$ sur $S$
induit une action
\footnote{Cette action est non-triviale et caractérisée par 
la structure de monoïde (non-triviale) de $\SSSS_1$.}
 de $A(1)$ sur $S(1)$
dans $\Diag_2(\Sch/S)$,
puis une action de 
$\CCd A(1)$ sur $\CCd S(1)$
dans $\Simp_1(\Diag_1 \Sch/S)$.

\subsubsection{}

On rappelle la construction du diagramme $\Ner$ 
par \cite{I} VI 2.5.1, VI 11.1.2.4
pour une action générale dans une catégorie $T$.
On suppose que $T$ possède des produits finis
et un objet final.
Soit $K$ un objet de $T$,
et $M$ un monoïde de $T$ qui agit sur $K$.
On lui associe l'objet simplicial
$$
\Ner(M,K)\in \Simp_1(T)
$$
dont on explique la définition.
Pour tout $n\in\NNN$,
son $n$-ième étage est 
$\Ner_n(M,K)=M^n\times K$.
On définit les morphismes
$\dd_i: \Ner_{n+1}(M,K)\to \Ner_n(M,K)$
par les formules
$$
\dd_0(a_0,...,a_n,x)=(a_1,...,a_n,a_0x)
$$
$$
\dd_i(a_0,...,a_n,x)=(a_0,...,a_{i}a_{i-1},...a_n,x),\quad
i=1,...,n
$$
$$
\dd_{n+1}(a_0,...,a_n,x)=(a_0,...,a_{n-1},x)
$$
pour tout $(a_0,...,a_n,x)\in M^{n+1}\times K$, 
et les $\sss_j: \Ner_{n}(M,K)\to \Ner_{n+1}(M,K)$ par
$$
\sss_j(a_0,...,a_{n-1},x)=(a_0,...,a_{j-1},1,a_j,...,a_{n-1},x),
\quad j=0,...,n
$$
pour tout $(a_0,...,a_{n-1},x)\in M^{n}\times K$.
Quand on prend $T=\Simp_1(\Diag_1 \Sch/S)$,
l'action de $\CCd A(1)$ sur $\CCd S(1)$ 
définit l'objet
$$
\Ner(\CCd A(1), \CCd S(1))
\in
\Simp_2(\Diag_1 \Sch/S).
$$

\subsubsection{}
Soit $X$ un multi-diagramme de schémas.
Dans cette section on définit une catégorie $\Mod(X)$
de systèmes de faisceaux sur $X$
qu'on utilise dans la définition \ref{nerd}.
La catégorie $\Mod(X)$ est similaire au
topos total $\Topd$ de  \cite{I} VI 5.2,
mais elle n'est pas un cas particulier de $\Topd$.

\begin{nota}
Pour tout morphisme de schéma $f:X\to Y$,
on note $f^{-1}$ le foncteur 
d'image inverse ensembliste,
et $f^*$ l'image inverse de $\OX$-module,
dans tous les topos: Zariski, fpqc, \etc.

\end{nota}

On remarque que dans \cite{I} VI, VII, 
Illusie utilise $f^*$ pour l'image inverse ensembliste,
parce qu'il travaille avec un grand site
(donc $f^{-1}=f^*$).
Dans notre cas,
nous travaillons avec un petit site
et $f^{-1}$ et $f^*$ ne sont pas égaux.
On note que les théorèmes dans \cite{I} VI 8,
qui sont écrits avec $f^*$,
ne sont valables que pour notre $f^{-1}$.

On donne la définition de $\Mod(X)$.
On considère le cas où $X$ est un $1$-diagramme,
avec une généralisation évidente sur les $r$-diagrammes pour tout $r$.
On suppose que $X$ est un $1$-diagramme de schémas de type $I$.
On note $\ob I$ l'ensemble des objets de $I$,
et $\fl I$ l'ensemble des flèches de $I$.
Pour toute flèche $f: i\to j$ de $I$,
on note $f: X_i\to X_j$
le morphisme de schémas associé.
On rappelle que $\Mod(\OO_{Y})$ est la catégorie
des faisceaux (Zariskiens) de $\OO_{Y}$-modules
pour tout schéma $Y$.

\begin{defi} \label{ModX}
On note $\Mod(X)$
la catégorie des systèmes
$(M_i, M_f)_{i\in \ob I, f\in\fl I}$,
où
$M_i$ appartient à $\Mod(\OO_{X_i})$
pour tout $i$, et
$M_f: M_i\to f^*M_j$
est un morphisme de $\Mod(\OO_{X_i})$
pour toute  $f: i\to j$,
tel que les $M_f$ vérifient 
la condition de cocycle évidente.
Les morphismes de $\Mod(X)$
sont les morphismes de système $(M_i)$
compatibles avec tous les $M_f$.
\end{defi}

On appelle $M_f$ les morphismes \emph{de transition}.
La catégorie $\Mod(X)$ est abélienne
lorsque tous les $f:X_i\to X_j$ sont plats.
Dans ce cas on note $\DD^*(X)$
la catégorie dérivée de $\Mod(X)$.

\subsubsection{}
On donne la définition du foncteur $\nerd$.
On note
$$\dd_0,\dd_1:
\CCd A(1)\times\CCd S(1)
\rightrightarrows
\CCd S(1)$$
les différentielles entre les deux derniers étages de 
$\Ner(\CCd A(1), \CCd S(1))$:
$\dd_0$ est induite par l'action de $\CCd A(1)$ sur $\CCd S(1)$,
et $\dd_1$ est la projection sur $\CCd S(1)$.
Alors pour tout objet $M$ de $\Mod(\ZZ[\CCd A(1)]\otimes\OS)$,
$M$ est un faisceau sur $\CCd S(1)$
tel que
la $\CCd A(1)$-structure de $M$
est équivalente à un morphisme de faisceau sur $\CCd A(1)\times\CCd S(1)$:
$$a_M: \dd_1^*M\to\dd_0^*M,$$
qui vérifie la condition de cocycle de \cite{I} VI 8.1.1.
Pour tout $n\in\NNN$, on note
$$\pp_n:
\CCd A(1)^n\times\CCd S(1)
\to
\CCd S(1)
$$
la projection  
sur $\CCd S(1)$.

\begin{defi} \label{nerd}
On définit le foncteur
$$
\nerd:
\Mod(\ZZ[\CCd A(1)]\otimes\OS)
\to
\Mod(\Ner(\CCd A(1), \CCd S(1)))
$$
comme suit.
Pour tout $M\in\Mod(\ZZ[\CCd A(1)]\otimes\OS)$ 
et tout $n\in\NNN$,
son $n$-ième étage 
est défini par
$$
\nerd_n(M)=\pp_n^*M.
$$
Les morphismes de transition 
sont définis de la manière suivante (\cite{I} VI 8.1.5): 
$\nerd_{n+1}(M) \to \dd_i^*\nerd_{n}(M)$
est égal à l'identité pour $i=1,...,n+1$,
et à $(\dd_2...\dd_{n+1})^*a_M$ pour $i=0$;
$\nerd_{n}(M) \to \sss_j^*\nerd_{n+1}(M)$
est égal à l'identité pour $j=0,...,n$.
\end{defi}

On note que $\nerd$ est un foncteur exact,
et on utilise la même notation pour son foncteur dérivé.
La proposition suivante est la clé pour définir les complexes de Lie
sur le topos de Zariski.
Sa démonstration est l'objectif de \S 4.6.

\begin{prop} \label{nerdfidele}
Le foncteur 
$$
\nerd:
\Db(\ZZ[\CCd A(1)]\otimes\OS)
\to
\Db(\Ner(\CCd A(1), \CCd S(1)))
$$
est pleinement fidèle,
et son image essentielle se compose des objets $Y\in\Db(\Ner(\CCd A(1), \CCd S(1)))$,
tels que pour tout $i\in\ZZ$,
l'objet $\HH^i(Y)$ de $\Mod(\Ner(\CCd A(1), \CCd S(1)))$
est dans l'image essentielle de 
$\nerd$ pour les $\Mod(...)$.
\end{prop}

\subsection{Preuve de la proposition \ref{nerdfidele}}

\subsubsection{}
Notre démonstration pour la proposition \ref{nerdfidele}
est une application (indirecte) d'un théorème d'Illusie
\cite{I} VI 8.4.2.1.
Pour rappeler ce théorème on fixe quelques notations.
On appelle $X$ un \emph{topos fibré simplicial}
(ou un topos fibré au-dessus de $\ddo$)
si $X$ consiste en la donnée suivante:
un topos $X_n$ pour tout $n\in\NNN$ ,
et les morphismes de topos
$$
\dd_i: X_{n+1} \to X_n 
\quad \quad \sss_j: X_n\to X_{n+1}
$$
pour $i=0,...,n+1$ et $j=0,...,n$,
tel que les $\dd_i, \sss_j$ vérifient la loi de composition 
dans $\ddo$.
Pour un topos fibré simplicial $X$,
on appelle son \emph{topos classifiant} $\BB X$ 
la catégorie des couples 
$(M,a_M)$,
où $M$ est un objet de $X_0$,
et $a_M: \dd_1^{-1}M \to \dd_0^{-1}M$
est un morphisme dans $X_1$
qui vérifie la condition de cocycle 
(\cf \cite{I} VI 8.1.1).
On appelle son \emph{topos total} $\Topd X$ 
la catégorie des systèmes 
$(M_n, M_f)_{n\in\NNN,f\in\fl\ddo}$,
où pour tout $n\in\NNN$,
$M_n$ est un objet de $X_n$,
et pour tout $f: [m]\to [n]$ dans $\ddo$, 
$$
M_f: M_m \to f^{-1}M_n
$$
est un morphisme de $X_m$ 
qui vérifie la condition de cocycle évidente.
On remarque que la définition de $\Topd$
se généralise à un topos fibré sur une catégorie quelconque.
Les mêmes formules que dans la définition \ref{nerd} (en remplaçant ${}^*$ par ${}^{-1}$)
définissent un foncteur
$$
\nerd: \BB X \to \Topd X.
$$

Soit $\OO$ un anneau de $\BB X$
tel que $a_{\OO}$ est un isomorphisme.
On note  $\Mod(\BB X)$ 
la catégorie 
des $\OO$-modules dans $\BB X$,
et $\Mod(\Topd X)$ la catégorie des 
$\nerd\OO$-modules dans $\Topd X$.
Le foncteur $\nerd$ induit naturellement un foncteur exact
$$
\nerd: \Mod(\BB X) \to \Mod(\Topd X).
$$

On rappelle deux conditions techniques 
dans l'énoncé de \cite{I}  VI 8.4.2.1.
Pour un topos fibré simplicial $X$,
on dit que $X$ est \emph{bon} si tous les 
$\dd_i^{-1}$ et $\sss_j^{-1}$
admettent des adjoints à gauche.
On dit que $X$ est une \emph{pseudo-catégorie}
si pour tout $n\in\NNN$,
le morphisme de changement de base
$$
(\dd_0...\dd_0)^{-1}{\dd_1}_* \to {\dd_{n+1}}_*(\dd_0...\dd_0)^{-1}
$$
relatif au carré
$$
\xymatrix{
X_{n+1}
\ar[r]^{\dd_0...\dd_0}
\ar[d]_{\dd_{n+1}}
&X_1
\ar[d]^{\dd_{1}}\\
X_n
\ar[r]^{\dd_0...\dd_0}
&X_0
}
$$
est un isomorphisme.
Le théorème suivant est démontré dans \cite{I}  VI 8.4.2.1.

\begin{theo}[Illusie] \label{thmillusie}
Soit $X$ un topos fibré simplicial,
et $\OO$ un anneau de $\BB X$ comme précédemment.
Supposons que $X$ est une bonne pseudo-catégorie,
telle que pour tout $n\in\NNN$,
$(\dd_0...\dd_0)^{-1}$
envoie les injectifs de $\Mod(X_0)$ 
sur des injectifs de $\Mod(X_n)$.
Alors 
$$
\nerd: \Db(\BB X) \to \Db(\Topd X)
$$
est pleinement fidèle,
et son image essentielle se compose des objets 
$Y\in\Db(\Topd X)$,
tels que pour tout $i\in\ZZ$,
$\HH^i(Y)$ 
est dans l'image essentielle du foncteur
$\nerd: \Mod(\BB X) \to \Mod(\Topd X)$.
\end{theo}

\subsubsection{}
On explique pourquoi la proposition \ref{nerdfidele}
n'est pas un corollaire direct du théorème \ref{thmillusie}.
L'idée naïve est d'utiliser le topos fibré simplicial $X$
tel que pour tout $n$,
$$X_n=\Topd{\Ner_n(\CCd A(1),\CCd S(1))}_{\Zar}$$
le $\Topd$ du petit topos de Zariski de 
$\Ner_n(\CCd A(1),\CCd S(1))$.
Le problème est que les définitions des morphismes de transition
pour les catégories $\Topd X$ et $\Mod(X)$ 
(\cf \S 4.6.1 et la définition \ref{ModX})
ne sont pas les mêmes 
(à cause de la différence entre $f^{-1}$ et $f^*$).
Donc avec ce choix de $X$,
$\Db(\Topd X)$ n'est pas la même catégorie que celle dans la proposition 
\ref{nerdfidele}.

Néanmoins, on observe que
$f^{-1}$ et $f^*$
coïncident lorsque $f$ 
appartient au site ambiant.
Un calcul direct à partir de la définition implique que
tous les morphismes
dans $\Ner(\CCd A(1), \CCd S(1))$
sont des morphismes dans $(\Sch/S)$
de la forme $A^m \times S \to A^n\times S$
pour certains $m$ et $n$.
Ces morphismes sont des isomorphismes locaux 
(au sens de \S 4.6.3)
dont le topos ambiant est équivalent au topos de Zariski.
Donc essentiellement 
on peut appliquer
le théorème \ref{thmillusie} 
au topos de Zariski
dans la proposition \ref{nerdfidele}
(via le site des isomorphismes locaux).

\subsubsection{}

On définit la notion  d'isomorphisme local.
On dit qu'un morphisme de schéma $f:Y\to X$ 
est un \emph{isomorphisme local},
s'il existe un recouvrement ouvert
$(V_i)_{i\in I}$ de $Y$,
tel que pour tout $i$,
la restriction de $f$ sur $V_i$
est une immersion ouverte.

\begin{defi} \label{loc}
Soit $X$ un schéma.
On définit le site $\Xloc$
de la manière suivante:
comme une catégorie,
on a 
$$
\Xloc
=
\{f:Y\to X \text{ isomorphisme local}
\};
$$
les familles couvrantes d'un objet $Y$ de $\Xloc$
sont toutes les familles
$(f_\alpha: Y_\alpha \to Y)_{\alpha\in A}$
de morphismes dans $\Xloc$
telles que 
$\mathop\bigcup\limits_{\alpha}
f_\alpha(Y_\alpha)
=Y$.
On appelle $\Xloc$ le \emph{site loc} de $X$.
\end{defi}

On note $\XZar$ 
le petit site de Zariski de $X$,
et $\Xloct,\XZart$ les topos associés
aux sites $\Xloc$ et $\XZar$, respectivement.
Le morphisme de site naturel $\XZar \to \Xloc$
définit un morphisme de topos
$$
\delta: \Xloct
\to 
\XZart.
$$
La donnée de $\delta$ consiste en un couple 
de foncteurs adjoints $(\delta^{-1},\delta_*)$,
où
$\delta_*:\Xloct \to \XZart$
est la restriction 
d'un faisceau sur $\Xloc$ à $\XZar$,
et
$\delta^{-1}:\XZart\to\Xloct$
est la faisceautisation d'un faisceau sur $\XZar$.
On montre que $\delta$ est une équivalence de topos.

\begin{prop} \label{loczariski}
Les foncteurs $\delta_*$ et $\delta^{-1}$
sont quasi-inverses entre eux.
\end{prop}

\begin{proof}
Comme tout objet de $\Xloc$
est recouvert par des objets de $\XZar$,
c'est un cas particulier de 
\cite{S4} III, théorème 4.1.
\end{proof}

Pour tout schéma $X$, 
on note $\OO_{\Xloc}$
le faisceau $\delta^{-1}\OX$ sur $\Xloc$,
et $\Mod(\OO_{\Xloc})$ 
la catégorie des $\OO_{\Xloc}$-modules.
On obtient

\begin{coro} \label{loczarmod}
Le foncteur 
$\delta_*: \Mod(\OO_{\Xloc}) \to \Mod(\OX)$
est une équivalence de catégorie.
\end{coro}

Soit $f: Y\to X$ un morphisme de schéma quelconque.
Alors $f$ induit un morphisme de site $\Xloc \to \Yloc$
par le changement de base via $f$,
et un morphisme de topos
 $f: \Yloct \to \Xloct$,
qui consiste en un couple de foncteurs adjoints
$$
f^{-1}: \Xloct \to \Yloct
\quad\quad 
f_*: \Yloct \to \Xloct,
$$
où $f_*$ est la restriction d'un faisceau
sur $\Yloc$ à $\Xloc$,
et $f^{-1}$ est défini d'une façon similaire 
à l'image inverse ensembliste usuelle 
pour le site de Zariski.
Lorsque $f$ est un isomorphisme local,
on a $f^{-1}\OO_{\Xloc}=\OO_{\Yloc}$,
donc $f^{-1}$ induit un foncteur  
$f^{-1}: \Mod(\OO_{\Xloc}) \to \Mod(\OO_{\Yloc})$.
On vérifie aisément que

\begin{lemm} \label{loczarf}
Si $f: Y\to X$ est un isomorphisme local,
alors le diagramme 
$$
\begin{CD}
\Mod(\OX)       @>f^*>>          \Mod(\OY)\\
@VV\delta^{-1}V        @VV\delta^{-1}V\\
\Mod(\OO_{\Xloc})   @>f^{-1}>>      \Mod(\OO_{\Yloc})
\end{CD}
$$
est commutatif.
\end{lemm}

\subsubsection{}
Nous finissons la démonstration de la proposition \ref{nerdfidele}.
On définit un topos fibré simplicial $X=(X_n)_{n\in\NNN}$, 
tel que pour tout $n\in\NNN$,
$$
X_n=\Topd{\Ner_n(\CCd A(1),\CCd S(1))}_{\lloc}
$$
le $\Topd$ du topos loc
de $\Ner_n(\CCd A(1),\CCd S(1))$,
et les $\dd_i, \sss_j$ pour $X$ sont induits par ceux pour 
$\Ner(\CCd A(1),\CCd S(1))$.
On note $\OO\in X_0$
le faisceau structurel de 
$\CCd S(1)_{\lloc}$,
qui est un objet de $\BB X$
parce que
$\dd_0^{-1}\OO$ et $\dd_1^{-1}\OO$
sont canoniquement isomorphes
au faisceau structurel
de $(\CCd A(1)\times\CCd S(1))_{\lloc}$.
Il existe un isomorphisme canonique
$$
\Topd X =  \Topd \Ner(\CCd A(1),\CCd S(1))_{\lloc},
$$
tel que
$\nerd\OO$ est égal au faisceau structurel
de $\Ner(\CCd A(1),\CCd S(1))_{\lloc}$.
On note 
la catégorie des $\nerd\OO$-modules par
$\Mod(\Ner(\CCd A(1),\CCd S(1))_{\lloc})$.

\begin{prop} \label{nerdloc}
Soit $X$ le topos fibré simplicial défini ci-dessus.
Alors 
$$
\nerd: 
\Db(\BB X) \to \Db(\Ner(\CCd A(1),\CCd S(1))_{\lloc})
$$
est pleinement fidèle,
et son image essentielle se compose des objets $Y$ 
tels que pour tout $i$,
$\HH^i(Y)$ est dans l'image essentielle du foncteur 
$\nerd$ pour les modules.
\end{prop}

\begin{proof}
Pour appliquer le théorème \ref{thmillusie}
on doit vérifier
\begin{enumerate}[label=\alph*)]
\item X est un bon topos fibré;
\item X est une pseudo-catégorie;
\item $(\dd_0...\dd_0 )^{-1}$ envoie 
les injectifs dans les injectifs.
\end{enumerate}

Les conditions a) et c) sont essentiellement un corollaire du fait que,
les $\dd_i$ et $\sss_j$ dans le diagramme $\Ner(\CCd A(1),\CCd S(1))$
sont des isomorphismes locaux (\cf \S 4.6.2),
donc ils sont des morphismes de localisation pour le site loc
(donc admettent des adjoints à gauche exacts).
La condition b) se déduite facilement de la propriété suivante:
dans le diagramme
$$
\begin{CD}
\CCd A(1)^{n+1}\times \CCd S(1)
       @>\dd_0...\dd_0>>          
\CCd A(1)\times \CCd S(1)\\
@VV\dd_{n+1}V        
@VV\dd_{1}V\\
\CCd A(1)^{n}\times \CCd S(1)
   @>\dd_0...\dd_0>>      
\CCd S(1),
\end{CD}
$$
la première ligne est une réunion disjointe
de la deuxième ligne
avec les morphismes verticaux comme les projections.
\end{proof}

\begin{proof}[Preuve de la proposition \ref{nerdfidele}]
Soit $X$ le topos fibré simplicial 
défini dans la proposition \ref{nerdloc}.
L'équation \cite{I} VI 11.5.3.2 implique un isomorphisme
entre $\Db(\BB X)$ et $\Db(\ZZ[\CCd A(1)]\otimes \OO_{\Sloc})$.
On en déduit un diagramme commutatif
$$
\begin{CD}
\Db(\ZZ[\CCd A(1)]\otimes \OO_{S})
       @>\nerd>>          
\Db(\Ner(\CCd A(1),\CCd S(1)))\\
@VV\delta^{-1}V        
@VV\delta^{-1}V\\
\Db(\ZZ[\CCd A(1)]\otimes \OO_{\Sloc})
   @>\nerd>>      
\Db(\Ner(\CCd A(1),\CCd S(1))_{\lloc}),
\end{CD}
$$
où la flèche $\nerd$ en haut (\resp en bas)
est le foncteur dans la proposition \ref{nerdfidele} 
(\resp la proposition \ref{nerdloc}).
Le diagramme est commutatif par le lemme \ref{loczarf},
et les flèches verticales sont
des équivalences par le corollaire \ref{loczarmod}.
Donc la proposition \ref{nerdfidele}
se déduit de la proposition \ref{nerdloc}.
\end{proof}

\begin{rema} \label{remarqueillusie}
Il semble que  la réduction via le site loc 
de la proposition \ref{nerdfidele}
au théorème \ref{thmillusie}  est élémentaire;
néanmoins le théorème \ref{thmillusie} lui-même
est non-trivial
(qui dépend de la théorie des modules induits et coinduits
dans \cite{I} VI 8.3).
Comme la proposition \ref{nerdfidele}
est d'une nature combinatoire,
il existe probablement une preuve directe
(sans utiliser les sites),
mais elle ne doit pas être simple 
à cause de la théorie de  \cite{I} VI 8.3.

\end{rema}

\subsection{Le complexe $\AlGv$}

\subsubsection{}
Dans cette section nous définissons la variante $\AlGv$
sur le topos de Zariski du complexe $\AllGv$ d'Illusie.
Nous rappelons d'abord le complexe cotangent usuel.
Soit $X\to Y$ un morphisme de schéma.
Son \emph{complexe cotangent} $\LL_{X/Y}$
est le $\OX$-module simplicial 
$$
\LL_{X/Y}=\Omega_{\PPd(\OX)/\OY}\otimes_{\PPd(\OX)}\OX,
$$
où $\PPd(\OX)$ est la résolution standard (\cite{I} I 1.5)
de $\OX$ sur $\OY$.
L'objet de $\DD^{\leq 0}(\OX)$ qui correspond
à $\LL_{X/Y}$ 
sera noté  par la même notation.

On rappelle que la notion de complexe cotangent s'étend aux morphismes
de multi-diagrammes de schémas par \cite{I} VII 1.2.1.
On rappelle la définition pour les morphismes de $1$-diagrammes.
Soit $X$ un $1$-diagramme de schémas
de type $I$.
On note $\XZart$ le petit topos de Zariski
de $X$.
\footnote{$\XZart$ est un topos fibré au-dessus de $I$:
c'est la donnée de topos $\widetilde{X_{i,\Zar}}$
pour tout $i\in\ob I$, 
avec les morphismes de topos 
$\widetilde{X_{i,\Zar}}\to \widetilde{X_{j,\Zar}}$
induit par ceux entre les schémas $X_i\to X_j$
pour tout $i\to j$ dans $I$.}
On considère le topos total $\Top\XZar$ de $\XZart$ (\cite{I} VI 5.2.1):
les objets de $\Top\XZar$  sont les systèmes 
$(M_i, M_f)_{i\in\ob I, f\in\fl I}$,
où pour tout $i$,
$M_i$ est un faisceau sur $X_{i,\Zar}$,
et pour tout $f:i\to j$,
$M_f: f^{-1}M_j\to M_i$
est un morphisme dans $\widetilde{X_{i,\Zar}}$
qui vérifie la condition de cocycle.
Le faisceau structurel $\OO_{X}$ de $\XZart$
est un objet de $\Top\XZar$,
et on peut considérer les $\OO_{X}$-modules
dans $\Top\XZar$.

On suppose que $X\to Y$ est un morphisme de $1$-diagramme de schémas.
Par définition, 
son complexe cotangent $\LL_{X/Y}$ est le $\OO_{X}$-module simplicial,
tel que pour tout $i$,
la restriction de $\LL_{X/Y}$ sur $X_i$
est égale à $\LL_{X_i/Y_{\alpha(i)}}$,
où $Y_{\alpha(i)}$ est le sommet de $Y$ 
au-dessous de $X_i$;
le morphisme de transition 
$f^{-1}\LL_{X_j/Y_{\alpha(j)}}
\to \LL_{X_i/Y_{\alpha(i)}}$
est défini
par la fonctorialité 
du complexe cotangent
(\cite{I} II 1.2.7.2).

Maintenant on rappelle la définition du complexe de Lie.
On considère d'abord le cas d'un seul schéma en groupes,
puis les diagrammes de schémas en groupes.
Soit $S$ un schéma,
et $G$ un schéma en groupes, 
plat et de présentation finie sur $S$.
Le \emph{complexe de co-Lie} de $G$ est 
l'objet de $\DD^{\leq 0}(\OS)$
défini par
$$
\ell_G = \LL e^* \LL_{G/S},
$$
où $e$ est la section unité de $G$.
On appelle 
$$
\ell_G^{\vee}
= \RR \uHom_{\OS}(\ell_G, \OS)
$$
le \emph{complexe de Lie} de $G$.

On étend les définitions de $\ell_G$ et $\ell_G^{\vee}$ 
aux diagrammes.
Soit $S$ un $1$-diagramme de schémas de type $I$.
Soit $G$ un schéma en groupes plat
et de présentation finie sur $S$.
\footnote{Par définition,
$G$ est un $1$-diagramme de schémas du même type que $S$,
tel que pour tout $i\in\ob I$, $G_i$ est un schéma en groupes plat
et de présentation finie sur $S_i$, 
et pour tout $i\to j$, 
le morphisme $G_i\to G_j$ est un morphisme de schémas en groupes (sur $S_j$).}
On définit
$$
\ell_G = \LL e^* \LL_{G/S} \in \DD^{\leq 0}(\OS)
$$
$$
\ell_G^{\vee}
= \RR \uHom_{\OS}(\ell_G, \OO_{S})
\in \DD^{\geq 0}(S).
$$
On rappelle que $\DD^{\geq 0}(S)$ est la catégorie définie par la définition \ref{ModX}.

\subsubsection{}
On donne la définition de $\AlGv$.
On suppose que $A$ est un anneau commutatif,
$S$ est un schéma,
et $G$ est un schéma en $A$-modules,
plat et de présentation finie sur $S$.
D'après la définition \ref{translation}
on a un 2-diagramme $G(1)$ de $(\Sch/S)$.
\footnote{
Contrairement à la description de \S 4.5.1 pour $S(1)$,
les composantes $G(1)^r$ de $G(1)$ ne sont pas des schémas $r$-simpliciaux constants 
mais bien des translations.}
Par la fonctorialité on a une action de $\CCd A(1)$ sur $\CCd G(1)$
dans $\Simp_1(\Diag_1 \Sch/S)$.
On considère le nerf
$$
\Ner(\CCd A(1), \CCd G(1))
\in
\Simp_2(\Diag_1 \Sch/S),
$$
et
on voit facilement que $\Ner(\CCd A(1), \CCd G(1))$
est un schéma en groupes 
sur $\Ner(\CCd A(1), \CCd S(1))$.
Le complexe  $\AlGv$ est défini par la proposition suivante.

\begin{prop} \label{AlGv}
Supposons que $S$ est \emph{plat} sur $\Spec\ZZ$.
Alors il existe un objet
$$
\AlGv \in \DD^{[0,1]}(A\otimes\OS)
$$
unique à isomorphisme unique près,
tel que 
$$
\nerd(\CCd{\AlGv}(1))
\cong 
\ell^{\vee}_{\Ner(\CCd A(1), \CCd G(1))}
$$
dans $\Db(\Ner(\CCd A(1), \CCd S(1)))$.
(On rappelle que les foncteurs $\CCCC$ et $\nerd$
sont définis dans les propositions \ref{Cun} et \ref{nerdfidele}.)
\end{prop}

\begin{proof}
Avec les propositions \ref{Cun} et \ref{nerdfidele},
c'est essentiellement la même démonstration que \cite{I} VII 4.1.4.1
pour $\AllGv$.
On montre d'abord que
$\ell^{\vee}_{\Ner(\CCd A(1), \CCd G(1))}$
est dans l'image essentielle de $\nerd$.
D'après la proposition \ref{nerdfidele},
c'est équivalent à montrer que les $\HH^i$ de 
$\ell^{\vee}_{\Ner(\CCd A(1), \CCd G(1))}$
sont dans l'image essentielle de $\nerd$ pour les modules.
Par \cite{I} VI 8.1.6, c'est équivalent au fait
que les $\HH^i$ de
$$
\Dec_1\ell^{\vee}_{\Ner(\CCd A(1), \CCd G(1))}
=
\ell^{\vee}_{\Dec_1\Ner(\CCd A(1), \CCd G(1))}
$$
sont cartésiens,
où $\Dec_1$ est le foncteur de décalé (\cite{I} VI 1.3).
Or, 
$\Dec_1\Ner(\CCd A(1), \CCd G(1))$
est cartésien au-dessus de $\Dec_1\Ner(\CCd A(1), \CCd S(1))$,
donc $\ell^{\vee}_{\Dec_1\Ner(\CCd A(1), \CCd G(1))}$
est quasi-cartésien par la compatibilité du complexe de Lie
avec le changement de base (\cite{I} VII 3.1.1.4).

On a vu qu'il existe
$
L \in \Db(\ZZ[\CCd A(1)]\otimes\OS)
$
tel que
$$
\nerd(L)\cong \ell^{\vee}_{\Ner(\CCd A(1), \CCd G(1))}.
$$
On montre que
$L$ est dans l'image essentielle de $\CCCC$.
D'après la proposition \ref{Cun},
il suffit de montrer la même propriété
pour l'image de $L$ dans $\Db(\CCd \SSSS_1,\OS)$.
Or, la définition de $L$ implique que 
son image dans $\Db(\CCd\SSSS_1,\OS)$
est isomorphe à $\ell^{\vee}_{\CCd G(1)}$.
D'après \cite{I} VII 4.1.3.1, on a 
$$
\ell^{\vee}_{\CCd G(1)} 
\cong 
\CCd\ell_G^{\vee}(1),
$$
donc l'hypothèse de la proposition \ref{Cun} est vérifiée.
\end{proof}

\begin{prop} \label{Glisse}
Soit $G$ un schéma en $A$-modules sur $S$ comme précédemment.
Si $G$ est lisse sur $S$, alors il existe un isomorphisme canonique
$$
\AlGv \cong \omega_G^{\vee}
$$
dans $\Db(A\otimes\OS)$,
où la structure de $A$-module de $\omega_G^{\vee}$  
est induite par celle de $G$.
\end{prop}

\begin{proof}
La lissité de $G$ implique 
$
\ell^{\vee}_{\Ner(\CCd A(1), \CCd G(1))}
\cong
\omega^{\vee}_{\Ner(\CCd A(1), \CCd G(1))},
$
donc la proposition se déduit de la pleine fidélité 
dans les propositions \ref{Cun} et \ref{nerdfidele}
et la définition de $\AlGv$.
\end{proof}

\begin{prop} \label{triangle}
Soit $A$ et $S$ comme précédemment, et
$$
0\to G\to B\to C\to 0
$$
une suite exacte de schémas en $A$-modules,
plats et de présentation finie sur $S$, 
telle que $B$ et $C$ sont lisses sur $S$.
Alors il existe un isomorphisme canonique
$$
\AlGv \cong [\omega_B^{\vee}\to \omega_C^{\vee}]
$$
dans $\Db(A\otimes\OS)$,
où le complexe à droite est placé en degrés $0$ et $1$.
\end{prop}

\begin{proof}
Le triangle fondamental du complexe de Lie
(\cite{I} VII 3.1.1.5) implique que
$$
\ell^{\vee}_{\Ner(\CCd A(1), \CCd G(1))}
\to \ell^{\vee}_{\Ner(\CCd A(1), \CCd B(1))}
\to \ell^{\vee}_{\Ner(\CCd A(1), \CCd C(1))}
\nrightarrow
$$
est distingué,
d'où un isomorphisme
$$
\ell^{\vee}_{\Ner(\CCd A(1), \CCd G(1))}
\cong [\omega^{\vee}_{\Ner(\CCd A(1), \CCd B(1))}
\to \omega^{\vee}_{\Ner(\CCd A(1), \CCd C(1))}]
$$
dans $\Db(\Ner(\CCd A(1), \CCd S(1)))$.
Comme $\AlGv$ et $[\omega_B^{\vee}\to \omega_C^{\vee}]$
sont de degré cohomologique $0,1$,
on peut utiliser la pleine fidélité de $\CCCC$ et $\nerd$
pour conclure.
\end{proof}

\begin{prop} \label{Res}
Soit $G/S$ un schéma en $A$-modules comme précédemment,
et $A'\to A$ un morphisme d'anneau. 
On note
$$
\Res: \Db(A\otimes\OS) \to \Db(A'\otimes\OS)
$$
le foncteur de restriction des scalaires induit par $A'\to A$.
Alors il existe un isomorphisme canonique
$$
\Res(\AlGv) {\cong} \AplGv
$$
dans $\Db(A'\otimes\OS)$.
\end{prop}

\begin{proof}
C'est la même propriété que \cite{I} VII 4.1.5 a)
pour $\AllGv$.
\end{proof}

\subsubsection{}
La proposition suivante ne sera utilisée que dans \S 5.4.

\begin{prop} \label{GGprime}
Soit $S$ un schéma plat sur $\Spec\ZZ$,
$A,A'$ deux anneaux,
et $G$ (\resp $G'$) un schéma en $A$-modules
(\resp schéma en $A'$-modules),
plat et de présentation finie sur $S$.
Munissons le schéma en groupes $G\times_S G'$
de la structure de $A\oplus A'$-module naturelle.
Alors il existe un isomorphisme canonique
$$
{_{A\oplus A'}\ell_{G\times G'}^{\vee}}
\cong {_{A}\ell_{G}^{\vee}} \oplus {_{A'}\ell_{G'}^{\vee}}
$$
dans $\Db((A\oplus A')\otimes\OS)$.
\end{prop}

On montre d'abord le lemme suivant.
On remarque que ce lemme reste valable dans un contexte 
plus général des diagrammes de schémas.

\begin{lemm} \label{XtimesY}
Soit $B$ le schéma de base.
On note $\times=\times_B, \otimes=\otimes_{\OO_B}$.
Soient $S, T$ deux schémas plats sur $B$,
et $X\to S, Y\to T$ deux morphismes plats
de schémas sur $B$.
Alors il existe un isomorphisme canonique
$$
(\LL_{X/S}\otimes \OY) \oplus 
(\OX \otimes \LL_{Y/T})
\xrightarrow{\sim} 
\LL_{X\times Y/S\times T}
$$
dans $\DD^{\leq 0}(\OO_{X\times Y})$.
\end{lemm}

\begin{proof}
Via le triangle fondamental du complexe cotangent,
c'est un corollaire de \cite{I} II 2.3.11,
le deuxième isomorphisme
(appliqué à $X\times Y$ et $S\times T$,
respectivement).
\end{proof}

\begin{proof}[Preuve de la proposition \ref{GGprime}]
Il suffit  d'appliquer le lemme \ref{XtimesY} aux produits fibrés
$$
\Ner(\CCd A(1),\CCd  G(1))
\tim{S_{\Ner(\CCd \SSSS_1, \CCd\SSSS_1)}}
\Ner(\CCd A'(1),\CCd G'(1))
=
\Ner(\CCd (A\oplus A')(1),\CCd (G\oplus G')(1))
$$
$$
\Ner(\CCd A(1),\CCd  S(1))
\tim{S_{\Ner(\CCd \SSSS_1, \CCd\SSSS_1)}}
\Ner(\CCd A'(1),\CCd S(1))
=
\Ner(\CCd (A\oplus A')(1),\CCd S(1))
$$
où $S_{\Ner(\CCd\SSSS_1, \CCd\SSSS_1)}$ 
est le diagramme 
de type $\Ner(\CCd\SSSS_1, \CCd\SSSS_1)$
de valeur constante $S$,
puis utiliser les définitions de 
${_{A}\ell_{G}^{\vee}}, {_{A'}\ell_{G'}^{\vee}}$
et ${_{A\oplus A'}\ell_{G\times G'}^{\vee}}$.
\end{proof}

\subsection{Le complexe $\lMKv$}
\subsubsection{}
Dans cette section on donne la  définition du complexe de Lie
monoïde-équivariant $\lMKv$ sur le topos de Zariski.
La définition de $\lMKv$ est similaire à celle de $\AlGv$ mais plus simple
(sans utiliser le foncteur $\CCCC$).
On suppose que $M$ est un monoïde commutatif,
$S$ est un schéma,
et $K/S$ est un schéma en groupes, plat et de présentation finie,
qui est muni d'une $M$-action.
On les associe les objets simpliciaux 
$\Ner(M,K), \Ner(M,S)$
de $(\Sch/S)$,
où $M$ agit trivialement sur $S$.
La  définition \ref{nerd} nous fournit
un foncteur
$$
\nerd: \Db(\ZM\otimes\OS) \to \Db(\Ner(M,S)),
$$
et la proposition \ref{nerdfidele} reste vraie dans ce contexte.
Par le même argument
que la proposition \ref{AlGv},
on obtient

\begin{prop} \label{lMKv}
Soit $K$ un schéma en groupes sur $S$ 
muni d'une $M$-action comme précédemment.
Alors il existe un objet $\lMKv$ de $\Db(\ZM\otimes\OS)$, 
unique à isomorphisme unique près,
tel que
$$
\nerd(\lMKv)\cong \ell^{\vee}_{\Ner(M,K)}
$$
dans $\Db(\Ner(M,S))$.
\end{prop}

\subsubsection{}
Lorsque $K$ s'injecte dans un schéma lisse 
par une immersion fermée régulière $M$-équivariante,
il existe une formule explicite pour $\lMKv$,
similaire à son analogue bien connu pour le complexe cotangent usuel
(\cite{I} III 3.2.6).
On omet la démonstration car c'est le même argument que
\cite{I} VII 2.2.5
pour le complexe cotangent équivariant.
On rappelle la notion d'une immersion régulière 
(dans le cas noethérien, et voir \cite{I} III 3.2.2 pour le cas général):
soit $i:Y\to X$ une immersion fermée de schémas
d'idéal $\III\subset \OX$,
alors $i$ est \emph{régulière} 
si localement sur $X$,
$\III$ est engendré par une suite régulière de $\OX$.
Pour tout complexe de $\OS$-modules $N$,
on note $N^{\vee}$ le complexe défini par
$N^{\vee}_i=\uHom(N_{-i}, \OS)$ pour tout $i\in\ZZ$.

\begin{prop} \label{intersectioncomplete}
Soit $K/S$ un schéma en groupes muni d'une $M$-action
comme précédemment.
Soit $X/S$ un schéma lisse muni d'une $M$-action,
et $i: K\to X$ une immersion fermée régulière $M$-équivariante sur $S$.
On note $\III \subset\OO_X$ l'idéal de $i$,
et $e$ la section unité de $K$.
Alors il existe un isomorphisme canonique
$$
\lMKv \cong [e^*(\III/\III^2) \to e^*i^*\Omega_{X/S}]^{\vee}
$$
dans $\Db(\ZM\otimes\OS)$,
où le complexe de droite est placé en degrés $0,1$
et muni de la  $M$-action naturelle 
(induite par les $M$-structures de $K$ et $X$).
\end{prop}

\subsection{La compatibilité des complexes de Lie}

Soit $A$ un anneau commutatif,
$S$ un schéma plat sur $\Spec\ZZ$, 
et $G$ un schéma en $A$-modules, plat et de présentation finie sur $S$.
On note $\Am$ le monoïde multiplicatif sous-jacent à $A$,
donc la $A$-structure sur $G$ induit une action de $\Am$ sur $G$.
Dans cette section
on montre une compatibilité entre les complexes 
$\AlGv$ et $\ell^{\Am,\vee}_G$
(voir les propositions \ref{AlGv} et \ref{lMKv} pour les définitions).
On note
$$
\Com: \Db(A\otimes \OS) \to \Db(\ZZ[\Am]\otimes\OS).
$$
le foncteur défini par le morphisme d'anneau naturel
$\ZZ[\Am]\to A$.
On a la proposition suivante,
qui est très simple mais ne figure pas dans \cite{I}.

\begin{prop} \label{compatibilite}
Soit $G$ un schéma en $A$-modules sur $S$
comme précédemment.
Alors il existe un isomorphisme canonique
$$
\Com(\AlGv)\cong \ell^{\Am,\vee}_G
$$
dans $\Db(\ZZ[\Am]\otimes\OS)$.
\end{prop}

\begin{proof}

Cette proposition résulte du fait que,
quand on restreint l'action de $\CCd A(1)$ sur $\CCd G(1)$
(dans la définition de $\AlGv$)
aux sommets de $0$-ièmes étages,
on retrouve l'action de $\Am$ sur $G$.
Plus précisément, on note
$\CCd S(1)$\ l'objet de $\Simp_1(\Diag_1 \Sch/S)$
défini dans \S 4.5.1.
On a
$$
\CC_0 S(1)= \mathop\bigsqcup\limits_
{r\in\NNN}
S(1)^{r},
$$
où $S(1)^{r}$ est l'objet constant de $\Simp_r(\Sch/S)$
de valeur $S$.
En prenant $r=0$,
on obtient un morphisme
$$
i_S: S\to \CCd S(1)
$$
dans la catégorie $\Diag_2(\Sch/S)$.
Le morphisme $i_S$ est équivariant 
par rapport à l'action triviale de $\Am$ sur $S$
et l'action de $\CCd A(1)$ sur $\CCd S(1)$.
Donc il induit un morphisme 
$$
i_S: \Ner(\Am, S) \to \Ner(\CCd A(1), \CCd S(1))
$$
dans $\Simp_1(\Diag_2 \Sch/S)$.

On considère le diagramme commutatif 
$$
\xymatrix{
\Db(A\otimes\OS)
\ar[r]^{\CCCC\qquad}
\ar[rd]^{\Com} 
&\Db(\ZZ[\CCd A(1)]\otimes\OS)
\ar[r]^{\nerd\quad}
\ar[d]^{i_S^*} 
&\Db(\Ner(\CCd A(1), \CCd S(1))) 
\ar[d]_{i_S^*}\\
&\Db(\ZZ[\Am]\otimes\OS)
\ar[r]^{\nerd\quad}
&\Db(\Ner(\Am, S))
}.
$$
Par définition de $\AlGv$ et $\ell^{\Am,\vee}_G$, on a 
$$
\nerd(\CCd {\AlGv}(1))
\cong 
\ell^{\vee}_{\Ner(\CCd A(1), \CCd G(1))}
$$
$$
\nerd(\ell^{\Am,\vee}_G)\cong 
\ell^{\vee}_{\Ner(\Am, G)}.
$$
Comme $\nerd$ est pleinement fidèle,
il suffit de montrer
$$
i_S^*(\ell^{\vee}_{\Ner(\CCd A(1), \CCd G(1))})
\cong \ell^{\vee}_{\Ner(\Am, G)}.
$$
Mais c'est essentiellement la compatibilité du complexe de Lie 
avec le changement de base
(\cite{I} VII 3.1.1.4):
on l'applique au diagramme cartésien 
$$
\begin{CD}
\Ner(\Am, G) @>i_G >> \Ner(\CCd A(1), \CCd G(1))\\
@VVV @VVV \\
\Ner(\Am, S) @>i_S >> \Ner(\CCd A(1), \CCd S(1))
\end{CD}
$$
où $i_G$ est défini de façon similaire de $i_S$
(qui envoie $G$ au premier sommet de $\CCd G(1)$).
\end{proof}


\section{Modèle local de niveau $\Gpp$}

\subsection{Le déterminant d'un complexe parfait}

\subsubsection{}

Nous rappelons la théorie du déterminant de Knudsen-Mumford \cite{KM}
qu'on utilisera plus loin dans la construction du modèle local.
Le déterminant est un foncteur 
qui transforme un isomorphisme (dans la catégorie dérivée)
entre complexes parfaits 
en un isomorphisme de fibrés en droites.
Plus précisément,
soit $S$ un schéma,
et $F$  un complexe de $\OS$-modules.
On dit que $F$ est un complexe \emph{parfait} (\cite{S6} I 4.7)
si localement sur $S$,
$F$ est quasi-isomorphe à un complexe borné de $\OS$-modules
libres de type fini.
On note 
$
\Parf(S) 
$
la sous-catégorie pleine de $\DD(\OS)$ 
qui consiste en tous les complexes parfaits,
et $\Parfis(S)$ la sous-catégorie de $\Parf(S)$
avec les mêmes objets mais seulement les isomorphismes.
On note $\Pic(S)$ la catégorie des fibrés en droites sur $S$,
et $\Picis(S)$ la sous-catégorie  de $\Pic(S)$
avec les mêmes objets mais seulement les isomorphismes.

Un \emph{déterminant} est la donnée $(\det, i)$,
où
$$
\det: \Parfis(S) \to \Picis(S)
$$
est un foncteur,
et pour tout suite exacte de complexes parfaits
$$
0\to G \xrightarrow{\alpha}F \xrightarrow{\beta} H\to 0,
$$
on l'associe un isomorphisme de fibrés en droites
(avec $\otimes = \otimes_{\OS}$)
$$
i_{\alpha, \beta}: \det G \otimes \det H \xrightarrow{\sim} \det F;
$$
la donnée $(\det, i)$ doit vérifier les axiomes i) à v)
de \cite{KM} definition 1.
Ici on ne répète pas les axiomes mais juste fait
quelques remarques:
l'axiome i) concerne la fonctorialité de $i_{\alpha, \beta}$
par rapport à la suite exacte;
l'axiome v) implique que,
si $F, G, H$ sont des $\OS$-modules localement libres de type fini,
alors $\det F=\bigwedge^{\max}F$
(et de même pour $G, H$),
et si l'on suppose que $F, G, H$ sont libres,
alors
$i_{\alpha, \beta}$ est défini par 
$$
(x_1\wedge...\wedge x_r)\otimes (y_1\wedge...\wedge y_s)
\mapsto x_1\wedge...\wedge x_r\wedge y_1\wedge...\wedge y_s,
$$
où $x_i$ (\resp $y_j$) est une base de $G$ (\resp $H$).
La donnée $(\det, i)$ est essentiellement déterminé par ces axiomes:
d'après  \cite{KM} theorem 1,
elle existe et est unique à isomorphisme unique près.

\subsubsection{}

Dans cet article, 
on ne considère que le déterminant des complexes parfaits d'un type particulier:

\begin{hyp} \label{etoile}
Soit $F$ un complexe de $\OS$-modules.
On dit que $F$ est de type $(\star)$ si
$F=[F_1 \xrightarrow{d_F} F_0]$
est concentré en degrés $-1,0$,
tel que $F_1, F_0$ sont des $\OS$-modules localement libres de type fini,
et $\rg F_1=\rg F_0$ sur chaque composante connexe de $S$.
\end{hyp}

On rappelle que,
si $F=[F_1 \xrightarrow{d_F} F_0]$ 
est de type $(\star)$,
alors la définition de \cite{KM} page 31 implique que
$$
\det F = (\det F_1)^{-1} \otimes \det F_0,
$$
avec la notion d'inverse $(\det F_1)^{-1}$ 
bien choisie pour éviter le problème des signes
(\cf \cite{KM} page 20, l'isomorphisme $\delta$).
La section de $\det F$ induite par $d_F$
est appelée la section \emph{canonique}.
Si $G=[G_1 \xrightarrow{d_G} G_0]$ est un autre complexe de type $(\star)$
(qui n'est pas nécessairement de même rang que $F$),
et
$
\lambda: F \to G
$
est un isomorphisme dans $\DD(\OS)$,
alors par la fonctorialité on obtient un isomorphisme de fibrés en droites
$\det(\lambda): \det F \xrightarrow{\sim} \det G$.
Nous renvoyons le lecteur à \cite{KM} page 26
pour la définition de $\det(\lambda)$
via le cylindre de $\lambda$.

On obtient les deux lemmes suivants. 
On omet leurs démonstrations car elles sont implicites dans \cite{KM}
(au moins dans le cas où $F$ et $G$ sont acycliques,
mais ces démonstrations sont vraies en toute généralité).
Pour le lemme \ref{fonctorialite},
voir la fonctorialité de $\det$ dans \cite{KM} page 30;
pour le lemme \ref{suiteexacte} 
voir le diagramme commutatif de \cite{KM} page 33.

\begin{lemm} \label{fonctorialite}
Soient $F, G$ deux complexes de $\OS$-modules  de  type $(\star)$,
et $\lambda: F\to G$ un isomorphisme dans $\DD(\OS)$.
Alors le diagramme
$$
\xymatrix{
\OS
\ar[r]^{d_F}
\ar[rd]_{d_G}
&\det F
\ar[d]^{\det(\lambda)}
_{\rotatebox{270}{$\sim$}}\\
&\det G
}
$$
est commutatif.
\end{lemm}

\begin{lemm} \label{suiteexacte}
Soient $F, G, H$ des complexes de $\OS$-modules de type
$(\star)$,
et
$$
0\to G\to F\to H\to 0,
$$
une suite exacte.
Alors le diagramme
$$
\xymatrix{
\OS\otimes\OS
\ar@{=}[r]
\ar[d]^{d_G\otimes d_H}
&\OS
\ar[d]^{d_F}\\
\det G \otimes\det H
\ar[r]^{\qquad\sim}_{\qquad i}
&\det F
}
$$
est commutatif.
\end{lemm}

\subsubsection{}

Soit $[F_1\xrightarrow{d_F}F_0]$
un complexe de $\OS$-modules de type $(\star)$.
On considère le complexe dual
$[F_0^{\vee}\xrightarrow{d_F^{\vee}}F_1^{\vee}]$
placé \emph{en degrés -1 et 0},
qui est aussi de type $(\star)$.
En identifiant
$(\det F_0^{\vee})^{-1}$
avec $\det F_0$,
on obtient un isomorphisme
$$
\psi: \det [F_1\xrightarrow{d_F}F_0]
\xrightarrow{\sim} 
\det [F_0^{\vee}\xrightarrow{d_F^{\vee}}F_1^{\vee}],
$$
où on échange les deux facteurs 
de $\det [F_1\xrightarrow{}F_0]$
via la loi de Koszul $x\otimes y \mapsto (-1)^{(\rg F)^2}y\otimes x$.
On vérifie aisément 
(en prenant les bases locales pour $F_0$ et $F_1$) que:

\begin{lemm} \label{complexedual}
Avec les notations précédentes,
le diagramme
$$
\xymatrix{
\OS 
\ar[r]^{d_F\qquad}
\ar[rd]_{d_F^{\vee}\quad}
&\det [F_1\to F_0]
\ar[d]^{\psi}_{\rotatebox{270}{$\sim$}}\\
&\det [F_0^{\vee}\to F_1^{\vee}]
}
$$
est commutatif.
\end{lemm}

\subsection{Complexe de Lie d'un schéma de Raynaud}

\subsubsection{}
On utilise les résultats de \S 4 
pour calculer le complexe de Lie ${_{\Fq}\ell_H^{\vee}}$
associé à un schéma en groupes de Raynaud $H$,
et répond la question de trouver les paramètres de Raynaud
dans \S 1.
On reprend les notations
de \S 2 et \S 3:
$F$ est un corps totalement réel de degré $n$,
$p$ est inerte dans $F$,
et on fixe un isomorphisme $\OF/p\cong \Fq$.
Le caractère de Teichmüller $\chi$
s'étend en un morphisme de monoïde multiplicatif
$\chi: \Fq \to \Zq$
avec $\chi(0)=0$.
On note $\Fr: \Fq\to\Fq$
le morphisme de Frobenius 
et $\Fr^{(k)}$ sa $k$-fois composée.
On obtient un diagramme commutatif
$$
\xymatrix{
&\Zq
\ar@{->>}[d]\\
\Fq 
\ar[ru]^{\chi^{p^k}}
\ar[r]^{\quad\Fr^{(k)}}
&\Fq
}
$$
pour $k=0,...,n-1$. 
Comme $\Spec \OF$ est étale sur $\Spec \ZZ$ en $p$,
 le critère infinitésimal implique que 
pour tout $k$,  il existe un unique morphisme
$\Fr^{(k)}: \OF\to \Zq$
qui est un relèvement de $\Fr^{(k)}:\Fq\to\Fq$
par le diagramme commutatif
$$
\xymatrix{
\OF 
\ar[r]^{\Fr^{(k)}}
\ar@{->>}[d]
&\Zq
\ar@{->>}[d]\\
\Fq
\ar[r]^{\Fr^{(k)}}
&\Fq
}.
$$

On définit pour tout $k$,
l'élément
$$
e_k=
\frac{1}{q-1}\mathop\sum\limits_{\lambda \in \Fqs}
[\lambda]\otimes \chi^{-p^k}(\lambda)
$$
de $\ZZ[\Fqm]\otimes\Zq$.
Alors $e_k$ est le projecteur sur la composante $\chi^{p^k}$-isotypique
d'une représentation de $\Fqm$.
D'autre part, on a des isomorphismes 
$$
\Fq\otimes\Zq =\Fq\otimes_{\Fp}\Fq
\xrightarrow{\sim}
\mathop\oplus\limits_{k=0}^{n-1}\Fq,
\quad  \OF\otimes\Zq \xrightarrow{\sim}
\mathop\oplus\limits_{k=0}^{n-1}\Zq
$$
définis par la même formule 
$x\otimes y \mapsto (xy,\Fr(x)y,...,\Fr^{(n-1)}(x)y)$.
On note (par un abus de notation)
$$e_k\in \Fq\otimes\Zq\qquad e_k\in \OF\otimes\Zq$$
les projecteurs sur les $k$-ièmes composantes.
Via les morphismes naturels
$$
\ZZ[\Fqm]\otimes\Zq \rightarrow \Fq\otimes\Zq 
\leftarrow \OF\otimes\Zq,
$$
on a la correspondance
$e_k \mapsto e_k \testleft e_k$.

On suppose que
$S$ est un schéma sur $\Spec \Zq$ qui est plat sur $\Spec\ZZ$, 
et que $H/S$ est un schéma en groupes de Raynaud.
On a un diagramme commutatif
$$
\xymatrix{
\Db(\ZZ[\Fqm]\otimes\OS)
  \ar[rd]^{e_k} 
&\Db(\Fq\otimes\OS)
\ar[r]^{\Res}
\ar[l]_{\quad\Com}
 \ar[d]^{e_k} 
&\Db(\OF\otimes\OS) 
 \ar[ld]_{e_k}\\
&\Db(\OS)
},
$$
où $\Res$ (\resp $\Com$) est défini par
la proposition \ref{Res}
(\resp la proposition \ref{compatibilite}),
et les flèches $e_k$ est la multiplication par $e_k$.
Les propositions \ref{lMKv}, \ref{AlGv}
définissent deux complexes
$\ell_{H}^{\Fqm,\vee} \in \Db(\ZZ[\Fqm]\otimes\OS)$ et 
${_{\OF}\ell_H^{\vee}} \in \Db(\OF\otimes\OS)$

\begin{prop} \label{compatibiliteRaynaud}
Pour $k=0,...,n-1$,
il existe un isomorphisme canonique
$$
e_k (\ell_{H}^{\Fqm,\vee})\cong e_k({_{\OF}\ell_H^{\vee}})
$$
dans $\Db(\OS)$.
\end{prop}

\begin{proof}
Par les propositions \ref{compatibilite} et \ref{Res}, on a 
$$
e_k (\ell_{H}^{\Fqm,\vee})
\cong e_k(\Com({_{\Fq}\ell_H^{\vee}}))
= e_k({_{\Fq}\ell_H^{\vee}})
= e_k(\Res({_{\Fq}\ell_H^{\vee}}))
\cong e_k({_{\OF}\ell_H^{\vee}}).
$$
\end{proof}

\subsubsection{}
Soit $S$ un schéma comme précédemment
et $H/S$ un schéma en groupes de Raynaud 
de paramètre $(\LLL_k, a_k, b_k)$.
On utilise les mêmes notations que dans \S 4.1.
On a une immersions fermée canonique $i: H\to X$,
où $H$ (\resp $X$) est définie par la $\OS$-algèbre
$\AAA$ (\resp $\SSS$).
On note $\III\subset\SSS$ l'idéal de $i$.
On rappelle que  $i$
est $\Fqm$-équivariante
et elle est une immersion régulière.
(On peut le vérifier par les coordonnées locales comme dans
la démonstration du lemme \ref{lemmedepart}.)

Comme $X/S$ est lisse, 
la proposition \ref{intersectioncomplete} implique qu'il existe un isomorphisme
$$
\ell_{H}^{\Fqm,\vee} \cong [e^*(\III/\III^2) \to e^*i^*\Omega_{X/S}]^{\vee}
$$
dans $\Db(\ZZ[\Fqm]\otimes\OS)$,
qui induit des isomorphismes de composantes isotypiques
$$
e_k (\ell_{H}^{\Fqm,\vee})\cong 
[e^*(\III/\III^2) \to e^*i^*\Omega_{X/S}]^{\vee}_{\chi^{p^k}}
$$
pour tout $k$.
D'après le lemme \ref{lemmedepart},
\footnote{
Ici on utilise la $\Fqm$-action,
mais c'est équivalente à la $\Fqs$-action:
tout $\Fqm$-module $V$ s'écrit comme une somme directe $V_0\oplus V_1$,
où $V_0$ est annulé par $0$,
et $V_1$ est un $\Fqm$-module trivial.
}  
on a 
$$
[\LLL_{k-1}^{\otimes p} \xrightarrow{a_k} \LLL_k ]^{\vee}
\cong
[e^*(\III/\III^2) \to e^*i^*\Omega_{X/S}]_{\chi^{p^k}}^{\vee},
$$
donc un isomorphisme 
$$
e_k (\ell_{H}^{\Fqm,\vee})\cong 
[\LLL_{k-1}^{\otimes p} \xrightarrow{a_k} \LLL_k ]^{\vee}
$$
dans $\Db(\OS)$.
Soit
$$
0\to H\to B\to B/H\to 0
$$
une résolution $\OF$-équivariante de $H$
par un schéma abélien $B/S$ muni d'une $\OF$-action
avec la condition de Kottwitz dans la définition \ref{Yzero}.
D'après la proposition \ref{triangle}, il existe un isomorphisme 
$$
{_{\OF}\ell_H^{\vee}} \cong [\omega_{B/H} \to \omega_B]^{\vee}
$$
dans $\Db(\OF\otimes\OS)$,
donc un isomorphisme
$$
e_k({_{\OF}\ell_H^{\vee}}) \cong [e_k\omega_{B/H} \to e_k\omega_B]^{\vee}
$$
dans $\Db(\OS)$.
La proposition \ref{compatibiliteRaynaud} implique que

\begin{coro} \label{Raynaudabelien}
Soit $S$ un schéma sur $\Spec\Zq$ qui est plat sur $\Spec\ZZ$, 
et $H/S$ un schéma en groupe de Raynaud
de paramètre $(\LLL_k, a_k, b_k)$.
Soit $B$ un schéma abélien qui résout $H$ comme précédemment.
Alors il existe un isomorphisme canonique 
$$
[\LLL_{k-1}^{\otimes p} \xrightarrow{a_k} \LLL_k ]
\cong
[e_k\omega_{B/H} \to e_k\omega_B]
$$
dans $\Db(\OS)$ pour $k=0,...,n-1$.
\end{coro}

\subsection{Les modèles locaux $M_0^+$ et $M_1^+$}

\subsubsection{}
Dans cette section
on définit le modèle local $M_1^+$ 
pour le schéma de Hilbert-Siegel $Y_1$ de niveau $\Gpp$.
La construction consiste en deux étapes:
d'abord on construit un $\mathbb{G}_m^{ng}$-torseur
$M_0^+$ au-dessus de $M_{0,\Zq}$,
puis on définit le modèle local $M_1^+$
par la même équation que dans la définition \ref{generateur},
tel que $M_1^+$  est 
fini ramifié sur $M_0^+$.
La nécessité de ce $\mathbb{G}_m^{ng}$-torseur
vient du fait qu'on connaît les fibrés en droites $\LLL_k$ 
sur $Y_1$,
mais que sur $M_{0,\Zq}$ 
on ne connaît que $\LLL_k^{\otimes(p-1)}$.

D'abord une remarque sur la définition du modèle local $M_0$.
On rappelle que $w_p$ est l'élément particulier de $\Zq$ défini dans \S 3.1.1.
On définit un schéma $M_{0,\Zq}'$ sur $\Spec\Zq$
par le même problème de module que dans la définition \ref{Mzero} pour $M_0$,
mais dans toutes les flèches $St_{i,\Zq}\to St_{i-1,\Zq}$
on remplace l'élément $p$ par $w_p$.
Comme $w_p\in p\Zqs$,
il existe un isomorphisme naturel 
$M_{0,\Zq}\cong M_{0,\Zq}'$.
Dans la suite on identifie $M_{0,\Zq}$ avec $M_{0,\Zq}'$
via cet isomorphisme.

On définit le $\mathbb{G}_m^{ng}$-torseur $M_0^+$.
Soit $S$ un schéma sur $\Spec\Zq$,
et $(W_0, \ldots, W_g)$ un $S$-point de $M_{0,\Zq}$:
pour tout $i$,
$W_i$ est un sous-$\OF\otimes\OS$-module de $\OF^{2g}\otimes\OS$
qui vérifie les hypothèses de la définition \ref{Mzero}.
On note pour $i=0,...,g$ et $k=0,...,n-1$,
$$W_{i,k}:=e_kW_i,$$
où $e_k\in\OF\otimes\Zq$
est le projecteur défini dans \S 5.2.1.
La définition \ref{Mzero} implique que
le $\OS$-module $W_{i,k}$ est 
localement un facteur direct de $\OS^{2g}$ de rang $g$.
On en déduit que $[W_{i,k} \to W_{i-1,k}]$
est un complexe de type $(\star)$.

\begin{defi} \label{Mzeroplus}
On définit le schéma $M_0^+$ au-dessus de $M_{0,\Zq}$, 
tel que pour tout schéma $S$ sur $\Spec\Zq$
et tout élément $W_{\centerdot}\in M_{0,\Zq}(S)$,
l'ensemble de $S$-points de $M_0^+$
au-dessus de $W_{\centerdot}$
est égal à l'ensemble des trivialisations de fibré en droites
$$
v_{i,k}:
\det[W_{i,k} \to W_{i-1,k}]
\xrightarrow{\sim}\OS,\quad
i=1,...,g, \ k=0,...,n-1.
$$
\end{defi}

\subsubsection{}

Nous expliquons la construction du modèle local $M_1^+$.
On note $d_{i,k}$ la différentielle dans le complexe
$[W_{i,k} \to W_{i-1,k}]$.
Comme $W_{i,k}$ est un facteur direct de $\OS^{2g}$,
le complexe 
$
[(\OS^{2g}/W_{i-1,k})^{\vee}\to (\OS^{2g}/W_{i,k})^{\vee}]
$
est aussi de type $(\star)$.
On note $d_{i,k}^{\vee}$ sa différentielle.
On définit pour tout $i$ et $k$, 
un accouplement canonique 
$$\can: 
\det[W_{i,k} \to W_{i-1,k}] \otimes 
\det[(\OS^{2g}/W_{i-1,k})^{\vee}\to (\OS^{2g}/W_{i,k})^{\vee}]
\xrightarrow{\sim}\OS
$$
par la composition des flèches
$$
\xymatrix{
\det[W_{i,k} \to W_{i-1,k}] \otimes 
\det[(\OS^{2g}/W_{i-1,k})^{\vee}\to (\OS^{2g}/W_{i,k})^{\vee}]
\ar[d]^{\id\otimes\psi}
_{\rotatebox{270}{$\sim$}}
\\
\det[W_{i,k} \to W_{i-1,k}] \otimes 
\det[\OS^{2g}/W_{i,k}\to \OS^{2g}/W_{i-1,k}]
\ar[d]^{i}_{\rotatebox{270}{$\sim$}}\\
\det[\OS^{2g}
\xrightarrow{\diag(1,...,w_p,...,1)}
\OS^{2g}]
},
$$
où $\psi$ est défini dans le lemme \ref{complexedual},
et $i$ est associé à la suite exacte
$$
\xymatrix{
0\ar[r]
&W_{i,k}\ar[r] \ar[d]^{d_{i,k}}
&\OS^{2g} \ar[r] \ar[d]
&\OS^{2g}/W_{i,k}\ar[r] \ar[d]^{d_{i,k}}
&0\\
0\ar[r] 
&W_{i-1,k}\ar[r]
&\OS^{2g} \ar[r]
&\OS^{2g}/W_{i-1,k}\ar[r]
&0
}.
$$
Pour tout élément $(W_i, v_{i,k})$ de $M_0^+(S)$,
il existe une unique trivialisation
$$
v_{i,k}^{\vee}:
\det[(\OS^{2g}/W_{i-1,k})^{\vee}\to (\OS^{2g}/W_{i,k})^{\vee}]
\xrightarrow{\sim}\OS
$$
telle que
$v_{i,k}\otimes v_{i,k}^{\vee}=\can$.
On note $v_{i,k}d_{i,k}$
(\resp $v_{i,k}^{\vee}d_{i,k}^{\vee}$)
la composition de $v_{i,k}$ (\resp $v_{i,k}^{\vee}$)
avec $d_{i,k}$ (\resp $d_{i,k}^{\vee}$);
ce sont des éléments de $\Gamma(S,\OS)$.

\begin{defi} \label{Munplus}
Le modèle local $M_1^+$ est le 
$M_0^+$-sous-schéma de $\mathbb{A}^{2g}_{M_0^+}$
défini par les équations  (de variables $s_1,...,s_g, t_1,...,t_g$) suivantes:
$$
s_i^{p^n-1}=
v_{i,0}d_{i,0}\cdot (v_{i,n-1}d_{i,n-1})^p\cdot
...\cdot(v_{i,1}d_{i,1})^{p^{n-1}},\quad i=1,...,g
$$
$$
t_i^{p^n-1}=
v_{i,0}^{\vee}d_{i,0}^{\vee}\cdot (v_{i,n-1}^{\vee}d_{i,n-1}^{\vee})^p\cdot
...\cdot(v_{i,1}^{\vee}d_{i,1}^{\vee})^{p^{n-1}},\quad i=1,...,g
$$
$$
s_1t_1=...=s_gt_g,
$$
où $v_{i,k}, v_{i,k}^{\vee}, d_{i,k},d_{i,k}^{\vee}$ 
sont des objets universels sur $M_0^+$.
\end{defi}

\begin{theo} \label{diagmodellocal}
Le schéma $M_1^+$ est un modèle local pour $Y_1$:
il existe un diagramme 
$$
Y_1 \xleftarrow{\pi'} Z_1^+ \xrightarrow{f'} M_1^+,
$$
où $Z_1^+$ est un schéma sur $\Spec\Zq$,
et $\pi', f'$ sont des morphismes lisses.
De plus, on a $\dim M_1^+=\dim Y_1 + ng$.
\footnote{Ici $\pi'$ et $f'$ n'ont pas la même dimension relative,
néanmoins c'est suffisant pour comparer les cycles proches.}
\end{theo}

\subsubsection{}
Le reste de \S 5.3 est consacré à démontrer le théorème \ref{diagmodellocal}.
On définit d'abord le schéma $Z_1^+$ et la projection $\pi'$.
Soit $S$ un schéma sur $\Zq$.
Prenons un $S$-point $(A, H_i, x_i, y_i)$
de $Y_1$ 
(\cf la définition \ref{Yun} pour la notation complète).
Pour tout $i$,
on note 
$(\LLL_{i,k}, a_{i,k}, b_{i,k})$
le paramètre de Raynaud de $H_i/H_{i-1}$.
On définit le schéma $Z_1^+$ comme  suit:
l'ensemble des $S$-points de $Z_1^+$
au-dessus de $(A, H_i, x_i, y_i)$
est égal à l'ensemble des systèmes
$(\phi, u_{i,k})$
tels que
\begin{enumerate}[label=--]
\item $\phi: \mathcal{H}^1(A/H_{\centerdot}) \xrightarrow{\sim} 
St_{\centerdot,\Zq}\otimes_{\Zq}\OS$
est un isomorphisme de chaînes polarisées;

\item 
$u_{i,k}: \LLL_{i,k} \xrightarrow{\sim} \OS$
est une trivialisation de fibré en droites
pour $i=1,...,g$ et $k=0,...,n-1$.
\end{enumerate}
Le morphisme $\pi': Z_1^+ \to Y_1$ est la projection vers $Y_1$,
qui est un torseur sous un schéma en groupes lisse
(la proposition \ref{chainenormalisee}).

\subsubsection{}
Pour définir le morphisme $f'$,
il suffit de définir pour tout $S$-point $(\phi, u_{i,k})$ de  $Z_1^+$,
son image par $f'$ dans $M_1^+(S)$.
D'après la définition \ref{Munplus},
$f'(\phi, u_{i,k})$ est de la forme $(W_i, v_{i,k}, s_i, t_i)$.
La définition pour $W_i$ est claire:
on pose pour tout $i$,
$$
W_i:=\phi(\omega_{A/H_i})
\subset \OF^{2g}\otimes\OS.
$$

\subsubsection{}
On définit les trivialisations $v_{i,k}$.
On note d'abord $A, H_i, \LLL_{i,k}$
les objets universels sur $Y_{0,\Zq}$.
Le corollaire \ref{Raynaudabelien} nous fournit 
des isomorphismes canoniques dans $\Pic(Y_{0,\Zq})$:
$$c_{i,k}:
\det[\LLL_{i,k-1}^{\otimes p}\xrightarrow{a_{i,k}}\LLL_{i,k}]
\xrightarrow{\sim}
\det[e_k\omega_{A/H_i}\to e_k\omega_{A/H_{i-1}}]
$$
$$
c_{i,k}^{\vee}:
\det[\LLL_{i,k-1}^{\vee, \otimes p}\xrightarrow{b_{i,k}^{\vee}}\LLL_{i,k}^{\vee}]
\xrightarrow{\sim}
\det[e_k\omega_{(A/H_{i-1})^{\vee}}\to e_k\omega_{(A/H_{i})^{\vee}}].
$$
On va utiliser les mêmes notations $c_{i,k}, c_{i,k}^{\vee}$
pour leurs changements de base
sur tout schéma au-dessus de $Y_{0,\Zq}$.

Soit maintenant
$(A,H_i,x_i,y_i,\phi,u_{i,k})$
un $S$-point de $Z_1^+$,
et $W_{i,k}$ le faisceau défini dans \S 5.3.4.
Alors
$\phi$ induit des isomorphismes 
$$\phi:
\det[e_k\omega_{A/H_i}\to e_k\omega_{A/H_{i-1}}]
\xrightarrow{\sim} \det[W_{i,k}\to W_{i-1,k}]
$$
$$\phi^{-1}:
\det[e_k\omega_{(A/H_{i-1})^{\vee}}\to e_k\omega_{(A/H_{i})^{\vee}}]
\xrightarrow{\sim}
\det[(\OS^{2g}/W_{i-1,k})^{\vee}\to (\OS^{2g}/W_{i,k})^{\vee}].
$$
On définit la trivialisation $v_{i,k}$ par l'équation
$
v_{i,k} \cdot \phi \cdot c_{i,k} = 
u_{i,k-1}^{\otimes (-p)}\otimes u_{i,k}
$.
Dans le lemme suivant,
 on va voir que la trivialisation duale $v_{i,k}^{\vee}$ 
vérifie une équation similaire.
La démonstration est une comparaison 
(via les lemmes \ref{fonctorialite}, \ref{suiteexacte} et \ref{complexedual})
entre les images de la section canonique.

\begin{lemm} \label{trivduale}
Soit $v_{i,k}$ est la trivialisation
définie ci-dessus,
et $v_{i,k}^{\vee}$ sa trivialisation duale
(\S 5.3.2).
Alors on a une égalité des isomorphismes
$$
v_{i,k}^{\vee} \cdot \phi^{-1} \cdot c_{i,k}^{\vee} 
= (u_{i,k}^{-1}\otimes u_{i,k-1}^{\otimes p})\cdot \psi,
$$
où 
$\psi$
est  l'isomorphisme défini dans le lemme \ref{complexedual}.
\end{lemm}

\subsubsection{}
On définit les racines $s_i, t_i$ (comme des éléments de $\Gamma(S,\OS)$)
 par les équations
$$
x_i=s_i\cdot u_{i,0} \qquad y_i=t_i\cdot u_{i,0}^{-1}.
$$
On montre que les $s_i, t_i$ vérifient les équations dans la définition \ref{Munplus}.
Comme $x_i$ est un générateur de Raynaud (\cf la définition \ref{generateur}),
on a 
\begin{equation*}
\begin{aligned}
s_i^{p^n-1}&=
u_{i,0}\cdot (a_{i,0}a_{i,n-1}^{\otimes p} ... 
a_{i,1}^{\otimes p^{n-1}})\cdot u_{i,0}^{\otimes (-p^n)}\\
&=
(u_{i,n-1}^{\otimes(-p)}\otimes u_{i,0})a_{i,0}\cdot
((u_{i,n-2}^{\otimes(-p)}\otimes u_{i,n-1})a_{i,n-1})^{p}\cdot
...\cdot 
((u_{i,0}^{\otimes(-p)}\otimes u_{i,1})a_{i,1})^{p^{n-1}},
\end{aligned}
\end{equation*}
où  pour tout $k$,
$(u_{i,k-1}^{\otimes(-p)}\otimes u_{i,k})a_{i,k}$
est défini par
la composition des flèches dans le diagramme
$$
\xymatrix{
\OS
\ar[d]_{a_{i,k}}
\ar[rd]^{d_{i,k}}
\\
\det[\LLL_{i,k-1}^{\otimes p}\to\LLL_{i,k}]
\ar[rd]_{u_{i,k-1}^{\otimes(-p)}\otimes u_{i,k}}
\ar[r]_{\sim}^{\phi c_{i,k}}
&\det[W_{i,k} \to W_{i-1,k}]
\ar[d]^{v_{i,k}}
\\
&\OS
}.
$$
D'après le lemme \ref{fonctorialite} et la définition de $v_{i,k}$,
ce diagramme est commutatif,
d'où l'équation de $s_i$
dans la définition \ref{Munplus}.
Un même argument sur le diagramme
$$
\xymatrix{
\OS
\ar[d]_{b_{i,k}^{\vee}}
\ar[rd]^{d_{i,k}^{\vee}}
\\
\det[\LLL_{i,k-1}^{\vee,\otimes p}\to\LLL_{i,k}^{\vee}]
\ar[rd]_{(u_{i,k}^{-1}\otimes u_{i,k-1}^{\otimes p}) \psi}
\ar[r]_{\sim\qquad\qquad}^{\phi^{-1} c_{i,k}^{\vee}\qquad\qquad}
&\det[(\OS^{2g}/W_{i-1,k})^{\vee} \to (\OS^{2g}/W_{i,k})^{\vee}]
\ar[d]^{v_{i,k}^{\vee}}
\\
&\OS
}
$$
implique l'équation de $t_i$;
la commutativité de ce diagramme est un corollaire du lemme \ref{trivduale}.

\subsubsection{}
On montre que $f'$ est lisse.
On considère le morphisme $f: Z_0\to M_0$ défini dans \S 2.2.2.
D'après la proposition \ref{grothendieckmessing}
il suffit de montrer que 
$
f': Z_1^+ \xrightarrow{} Z_0\times_{M_0}M_1^+
$
est étale.
En fait, 
$Z_0 \times_{M_0}M_1^+$
est le schéma de module des systèmes
$
(A,H_., \phi, v_{i,k}, s_i, t_i)
$,
où $v_{i,k}$ est une trivialisation 
$\det[e_k\phi(\omega_{A/H_i})\to e_k\phi(\omega_{A/H_{i-1}})]$,
et $s_i,t_i$ sont définis par la définition \ref{Munplus}.
On définit un schéma $Z_0^+$ sur $\Spec\Zq$,
qui est le schéma de module des systèmes 
$
(A,H_i,\phi,u_{i,k})
$,
où les $u_{i,k}$ sont définies dans \S 5.3.3.
On obtient un diagramme cartésien
$$
\xymatrix{
Z_1^+
\ar[r]
\ar[d]^{f'}
&Z_0^+
\ar[d]^{h}\\
Z_0\times_{M_0}M_1^+
\ar[r]
&
Z_0\times_{M_0}M_0^+
}.
$$
Il suffit de montrer que $h$
est étale.
En fait, $h$ envoie
$(A,H_i,\phi,u_{i,k})$
sur $(A,H_i,\phi,v_{i,k})$,
tel que $v_{i,k}$ est associé à  $u_{i,k}$
de la même manière de \S 5.3.5.
Localement sur la base,
$h$ est égal à
$
\mathbb{G}_m^{ng} \to \mathbb{G}_m^{ng} 
$,
$
(u_{i,k}) \mapsto (u_{i,k-1}^{-p} u_{i,k}),
$
donc il est étale en prenant son algèbre de Lie.

\subsection{Le cas non-ramifié général}

Dans cette section, 
on étend la construction de $M_1^+$ dans \S 5.3
au cas non-ramifié général.
On suppose que $F/\QQ$ est un corps totalement réel de degré $n$,
et $p$ un nombre premier non-ramifié dans $F$,
tel que le cardinal résiduel de $p$ dans $\OF$
est égal à $q=p^{n'}$
(donc $p$ est décomposé comme un produit de 
$n/n'$ idéaux premiers).
On considère d'abord la construction de $Y_1$
dans \S 3.2.
Soit $S$ un schéma sur $\Spec\Zq$,
et $(A, H_i)$ un $S$-point de $Y_0$. 
Pour $i=1,...,g$,
le schéma en groupes $H_i$ est muni d'une action de l'algèbre
$\OF/p=\mathop\bigoplus\limits_{r=1}^{n/n'}\Fq$,
donc décomposé comme le produit
$$H_i=\mathop\prod\limits_r H_{i,r}.$$
Pour tout $(i,r)$,
$H_{i,r}/H_{i-1,r}$ est un schéma en groupes de Raynaud.
La définition de générateur de Raynaud s'étend à $H_i/H_{i-1}$:
un générateur $z_i$ de $H_i/H_{i-1}$ est la donnée
$$
z_i=(z_{i,1},...,z_{i,n/n'})
$$
où pour tout $r$,
$z_{i,r}$ est un générateur de Raynaud de $H_{i,r}/H_{i-1,r}$ .
Donc on peut définir le modèle entier $Y_1$ de la même façon 
que dans la définition \ref{Yun}.

Pour définir le modèle local $M_1^+$,
il suffit d'expliquer pourquoi
on peut trouver
les paramètres de Raynaud de chaque
schéma en groupes $H_{i,r}/H_{i-1,r}$
dans le même complexe $[\omega_{A/H_i}\to \omega_{A/H_{i-1}}]$.
C'est réalisé par la proposition \ref{GGprime}.
En effet, d'après la proposition \ref{triangle} 
il existe un isomorphisme
$$
{_{\OF}\ell_{H_i/H_{i-1}}^{\vee}} \cong
[\omega_{A/H_i}\to \omega_{A/H_{i-1}}]
$$
dans $\Db(\OF\otimes \OS)$.
Par la proposition \ref{Res}, ${_{\OF}\ell_{H_i/H_{i-1}}^{\vee}}$
est isomorphe à l'image de 
$$
{_{\OF/p}\ell_{H_i/H_{i-1}}^{\vee}}
\in \Db(\OF/p\otimes \OS)
$$
par le foncteur $\Res$.
La proposition \ref{GGprime} implique qu'il existe une décomposition
$$
{_{\OF/p}\ell_{H_i/H_{i-1}}^{\vee}} \cong
\mathop\bigoplus\limits_{r=1}^{n/n'}
{_{\Fq}\ell_{H_{i,r}/H_{i-1,r}}^{\vee}}
$$
dans $\Db(\OF/p\otimes \OS)$.
Si l'on note 
$\varepsilon_r$ le $r$-ième projecteur
dans les anneaux
$\OF\otimes\Zpp=\mathop\oplus\limits_r\Zq$
et $\OF/p=\mathop\oplus\limits_r\Fq$
pour tout $r$,
alors on a 
$$
\Res({_{\Fq}\ell_{H_{i,r}/H_{i-1,r}}^{\vee}}) \cong
\varepsilon_r\Res({_{\OF/p}\ell_{H_i/H_{i-1}}^{\vee}}) \cong
\varepsilon_r ({_{\OF}\ell_{H_i/H_{i-1}}^{\vee}}) \cong
\varepsilon_r[\omega_{A/H_i}\to \omega_{A/H_{i-1}}]
$$
dans $\Db(\Zq\otimes_{\Zpp}\OS)$.
Par la preuve de la proposition \ref{compatibiliteRaynaud} et le corollaire \ref{Raynaudabelien},
on sait que le $k$-ième paramètre de Raynaud de $H_{i,r}/H_{i-1,r}$
est isomorphe à
$e_k\varepsilon_r [\omega_{A/H_i}\to \omega_{A/H_{i-1}}]$,
où  $e_k$ est le projecteur défini dans \S 5.2.

Donc on peut utiliser les mêmes idées de \S 5.3 
pour définir les schémas $M_0^+$ et $M_1^+$
(pour définir $M_0^+$ 
il suffit de trivialiser tous les fibrés
$\det[e_k\varepsilon_r W_i\to e_k\varepsilon_r W_{i-1}]$),
et l'analogue du théorème \ref{diagmodellocal} reste valable.

\end{document}